\documentclass[11pt,a4paper]{article}

\usepackage{titlesec}
\usepackage{fancyhdr}
\usepackage{a4wide}
\usepackage{graphicx}
\usepackage{float}
\usepackage{amssymb}
\usepackage{amsmath}
\usepackage{amsthm}
\usepackage{color}
\usepackage{mathrsfs}
\usepackage{array}
\usepackage{eucal}
\usepackage{tikz}
\usepackage{multido}
\usepackage{wrapfig}
\usepackage[T1]{fontenc}
\usepackage{inputenc}
\usepackage[english]{babel}

\usepackage{lmodern}
\usepackage{hyperref}
\usepackage{geometry}
\usepackage{changepage}
\changepage{0pt}{}{}{}{}{0pt}{}{0pt}{10pt}
\usepackage[numbers]{natbib}
\usepackage{hyperref}
\hypersetup{
pdfpagemode=none,
pdftoolbar=true,        
pdfmenubar=true,        
pdffitwindow=false,     
pdfstartview={Fit},    
pdftitle={From zero transmission to trapped modes in waveguides},    
pdfauthor={L. Chesnel, V. Pagneux},     
pdfsubject={},  
pdfcreator={L. Chesnel, V. Pagneux},   
pdfproducer={L. Chesnel, V. Pagneux}, 
pdfkeywords={}, 
pdfnewwindow=true,      
colorlinks=true,       
linkcolor=magenta,          
citecolor=red,        
filecolor=cyan,      
urlcolor=blue           
}

\newcommand{\dsp}{\displaystyle}

\newcommand{\eps}{\varepsilon}
\newcommand{\Om}{\Omega}
\newcommand{\om}{\omega}
\newcommand{\mrm}[1]{\mathrm{#1}}

\newcommand{\Cplx}{\mathbb{C}}
\newcommand{\N}{\mathbb{N}}
\newcommand{\R}{\mathbb{R}}

\newcommand{\mH}{\mrm{H}}

\newcommand{\loc}{\mbox{\scriptsize loc}}

\newtheorem{theorem}{Theorem}[section]
\newtheorem{lemma}{Lemma}[section]
\newtheorem{remark}{Remark}[section]

\newtheorem{proposition}{Proposition}[section]

\begin{document}

~\vspace{0.0cm}
\begin{center}
{\sc \bf\LARGE  From zero transmission to trapped \\[6pt] modes in waveguides}
\end{center}

\begin{center}
\textsc{Lucas Chesnel}$^1$, \textsc{Vincent Pagneux}$^2$\\[16pt]
\begin{minipage}{0.96\textwidth}
{\small
$^1$ INRIA/Centre de math\'ematiques appliqu\'ees, \'Ecole Polytechnique, Universit\'e Paris-Saclay, Route de Saclay, 91128 Palaiseau, France;\\
$^2$ Laboratoire d'Acoustique de l'Universit\'e du Maine, Av. Olivier Messiaen, 72085 Le Mans, France.\\[10pt]
E-mails: \texttt{lucas.chesnel@inria.fr}, \texttt{vincent.pagneux@univ-lemans.fr}\\[-12pt]
\begin{center}
-- \today --
\end{center}
}
\end{minipage}
\end{center}
\vspace{0.4cm}

\noindent\textbf{Abstract.} 
We consider the time-harmonic scattering wave problem in a 2D waveguide 
at  wavenumber $k$ such that one mode is propagating in the far field. 
In a first step, for a given $k$, playing with one scattering branch of finite length, 
we demonstrate how to construct geometries with zero transmission. 
The main novelty in this result  is that the symmetry of the geometry is not needed : the proof relies on the unitary structure of the scattering matrix. 
Then, in a second step, from a waveguide with zero transmission, we show how to build geometries supporting trapped modes associated with eigenvalues embedded in the continuous spectrum. For this second construction, using the augmented scattering matrix and its unitarity, we play both with the geometry and the wavenumber. Finally, the mathematical analysis is supplemented by numerical illustrations of the results. \\

\noindent\textbf{Key words.} Waveguides, zero transmission, trapped modes, scattering matrix, asymptotic analysis.

\section{Introduction}

In this work, we are interested in the propagation of waves in a waveguide with two open channels (channels 1 and 2, see 
Figure \ref{DomainOriginal2D}). 
We work at a wavenumber $k$ sufficiently small so that only one wave can propagate in each of the channels. 
For the Neumann problem that we consider in the following, which appears for example for acoustics, for water-waves  or for electromagnetism, this wave is nothing but the planar wave. To describe the scattering process of the incident planar wave propagating in channel 1, classically one
introduces two complex coefficients, namely the \textit{reflection} and \textit{transmission coefficients}, denoted
$s_{11}$ and $s_{12}$, such that $s_{11}$ (resp. $s_{12}$) corresponds to the amplitude of the scattered field at infinity in channel 1 (resp. 2) (see (\ref{DefScatteringCoeff})). According to the energy conservation, we have
\[
|s_{11}|^2 + |s_{12}|^2 = 1.
\]
In the first part of the article, we are interested in constructing non-symmetric geometries such that the transmission coefficient $s_{12}$ is zero. In this case, all the energy is backscattered in channel 1, as if the waveguide was obstructed. This problem of \textit{zero transmission} is interesting by itself and seems to have received little attention in literature, especially from a theoretical point of view. The motivation for designing domains where $s_{12}=0$ is also linked to the second part of the paper where we show how to use them to find waveguides supporting so called \textit{trapped modes} associated with eigenvalues embedded in the continuous spectrum. We remind the reader that trapped modes are non zero solutions of the homogeneous wave Problem (\ref{PbInitial}) which are of finite energy. In particular, using Fourier decomposition, one shows that trapped modes are exponentially decaying at infinity. Unlike the zero transmission problem, the question of existence of trapped modes has been widely investigated in literature (see e.g. \cite{Urse51,Evan92,EvLV94,DaPa98,LiMc07,Naza10c,pagneux2013trapped}). Note that eigenvalues embedded in the continuous spectrum are also often called bound states in the continuum (BSCs or BICs) in quantum mechanics or in optics (see for example \cite{SaBR06,Mois09,GPRO10,zhen2014topological,gomis2017anisotropy} as well as the recent review \cite{HZSJS16}). \\
\newline 
One way to exhibit situations where $s_{12}=0$ is to use the so-called Fano resonance (see the seminal paper \cite{Fano61} and the review article \cite{MiFK10}). Let us present briefly the idea which is developed and justified in \cite{ShVe05,ShTu12,ShWe13,AbSh16,ChNaSu}. If trapped modes exist at a given wavenumber $k_0$ such that $k_0^2$ is an eigenvalue embedded in the continuous spectrum, then perturbing slightly the geometry allows one to exhibit settings where the scattering coefficients have a fast variation for $k$ moving on the real axis around $k_0$. And with an additional assumption of symmetry of the geometry, one can prove that $|s_{12}|$ passes exactly through $0$ and $1$. This process requires to start from a setting where it is known that there are trapped modes. We emphasize that trapped modes associated with eigenvalues embedded in the continuous spectrum are rather rare objects. In general, for a given geometry, the set of wavenumbers such that they exist is empty. Moreover, they are relatively unstable. If eigenvalues embedded in the continuous spectrum exist in a given setting, a small perturbation of the geometry will transform them into complex resonances as shown in \cite{AsPV00} (except if the perturbation is carefully chosen, see \cite{Naza13}).\\
\newline
In the present article, in order to create geometries where $s_{12}=0$, we follow another approach introduced in \cite{ChNPSu,ChPaSu}. As in \cite{ChNPSu,ChPaSu}, we assume that the waveguide is endowed with a branch of finite length $L-1$ (see Figure \ref{DomainOriginal2D} left). The scattering coefficients depend on the geometry, in particular on $L$. Computing an asymptotic expansion of $s_{11}=s_{11}(L)$, $s_{12}=s_{12}(L)$ as $L\to+\infty$ allows us to know the main behaviour of the complex curves $L\mapsto s_{11}(L)$, $L\mapsto s_{12}(L)$. When the geometry has some symmetries, it is shown in \cite{ChNPSu,ChPaSu} that asymptotically $L\mapsto s_{11}(L)$, $L\mapsto s_{12}(L)$ pass through zero periodically. Then, owing to the symmetry assumption, we can get more and exhibit geometries where \textit{zero reflection} ($s_{11}=0$, $|s_{12}|=1$), zero transmission ($|s_{11}|=1$, $s_{12}=0$) or \textit{perfect invisibility} ($s_{11}=0$, $s_{12}=1$) occur (not only asymptotically but exactly). In the present work, quite surprisingly, we prove that 
without this assumption of symmetry of the geometry, asymptotically $L\mapsto s_{12}(L)$ still passes through zero periodically. And moreover, we show that the unitary structure of the scattering matrix defined in (\ref{DefMatrixScattering}) is sufficient to conclude that $L\mapsto s_{12}(L)$ passes through zero exactly. We emphasize that without the assumption of symmetry, in general the curve for the reflection coefficient $L\mapsto s_{11}(L)$ does not pass through zero as $L\to+\infty$ (not even asymptotically). Thus, in the non-symmetric case we consider, we demonstrate zero transmission but zero reflection and so perfect invisibility, cannot be achieved. \\
\newline
In order to construct geometries supporting trapped modes, we will use a similar idea as before  but with a useful auxiliary object called the augmented scattering matrix  \cite{NaPl94bis,KaNa02,Naza06,Naza11} instead of the usual scattering matrix. To establish the existence of trapped modes, we will see that it is sufficient that a coefficient of the augmented scattering matrix, which is unitary, passes through the point of affix $-1+0i$. The unitary structure of the matrix will not be sufficient to obtain the desired result. But playing also a bit with the wavenumber, we will be able to circumvent the difficulty.\\
\newline
In the nicely written introduction of \cite{HeKN12}, from \cite{GPRO10}, the authors distinguish three main mechanisms to produce trapped modes associated with eigenvalues embedded in the continuous spectrum. The most common one is based on the decoupling between resonances and radiation modes due to symmetries of the problem \cite{Park66,Park67,EvLi91,EvLV94}. The second mechanism relies on the destructive interference between two resonances of a single resonator. It has been revealed by Friedrich and Wintgen in \cite{FrWi85} (see also \cite{SaBR06} and the review paper \cite{OkPR03}). The third idea to obtain trapped modes consists in using two resonant structures acting as a pair of mirrors. Tuning correctly the distance between the two mirrors, one can trap waves between them, the structure then being equivalent to a Fabry-Perot cavity. This technique has been used in \cite{McIv96,KuMc97,Port02,OrNK06,CaLo08,Sato12}. The latter mechanism is the one we will exploit in the present article. On the other hand, we shall limit ourselves to Helmholtz equation with Neumann boundary conditions. Dirichlet boundary conditions (for quantum waveguides for example) can be treated completely similarly. Thus, we can construct quantum waveguides supporting trapped modes. \\
\newline  
All through the article, for the sake of simplicity, we shall restrict to a very simple $2\mrm{D}$ waveguide. However, the techniques can be extended in a straightforward way to higher dimension and more complex geometries (see the discussion in Section \ref{SectionConclusion}). We will consider scattering problems in $\sf{T}$-shaped waveguides. In \cite{Naza10a,NaSh11}, the authors study the existence of discrete spectrum (spectrum below the continuous spectrum) for Dirichlet problems in such geometries. Let us mention also that the present work shares connections with \cite{SaBR06,DKLM07,BuSa11,HeKN12}. In the latter papers, the authors investigate the existence of trapped modes associated with eigenvalues 
embedded in the continuous spectrum in geometries similar to ours via numerical investigations or simplified models.\\
\newline
The outline is as follows. In Section \ref{SectionPerfectReflectivity}, for a given wavenumber below the first positive thresholds (monomode case), we explain how to construct geometries where the transmission coefficient is zero. In Section \ref{SectionTrappedModes}, we show how to design situations, playing both with the geometry and the wavenumber, such that trapped modes exist. In Section \ref{SectionNumExpe}, we provide numerical illustrations of the results. In the conclusion, we discuss possible generalizations of the present approach as well as open questions. Finally, in a short Annex, we detail the proofs of two auxiliary results used in the analysis. The main results of this work are Theorem \ref{MainThmPart1} (zero transmission) and Theorem \ref{MainThmPart2} (existence of trapped modes).

\section{Zero transmission}\label{SectionPerfectReflectivity}
In this first section, for a given wavenumber we explain how to construct waveguides where the transmission coefficient is zero. 

\subsection{Setting }

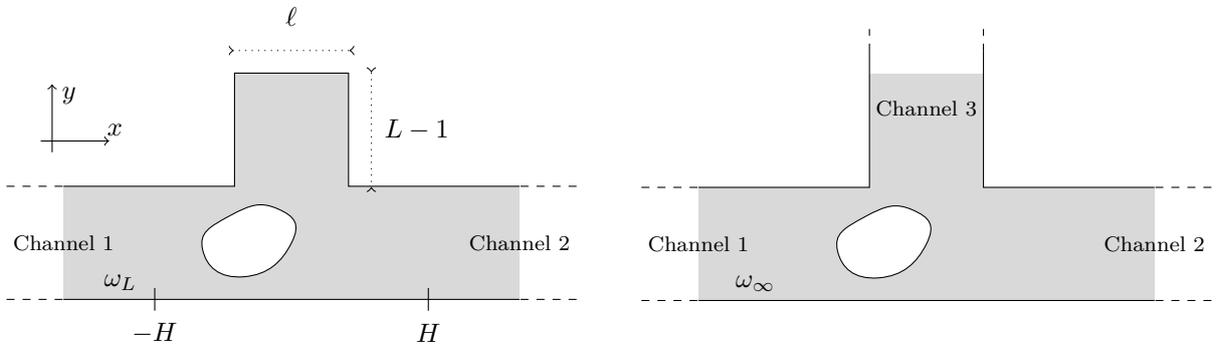
\begin{figure}[!ht]
\centering
\begin{tikzpicture}[scale=1.5]
\draw[fill=gray!30,draw=none](-2,1) rectangle (0.5,2);
\draw[fill=gray!30,draw=none](-3.5,1) rectangle (-2,2);
\draw[fill=gray!30,draw=none](-2,1.8) rectangle (-1,3);
\draw (-3.5,1)--(0.5,1); 
\draw (-3.5,2)--(-2,2)--(-2,3)--(-1,3)--(-1,2)--(0.5,2); 
\draw[dashed] (1,1)--(0.5,1); 
\draw[dashed] (1,2)--(0.5,2);
\draw[dashed] (-4,2)--(-3.5,2);
\draw[dashed] (-4,1)--(-3.5,1);
\draw[dotted,>-<] (-0.8,1.95)--(-0.8,3.05);
\draw[dotted,>-<] (-2.05,3.2)--(-0.95,3.2);
\begin{scope}[scale=0.6,yshift=1.5cm,xshift=-3.2cm]
\draw [fill=white] plot [smooth cycle, tension=1] coordinates {(-0.6,0.9) (0,0.5) (0.7,1) (0.5,1.5) (-0.2,1.4)};
\end{scope}
\node at (-0.4,2.5){\small $L-1$};
\node at (-1.5,3.5){\small $\ell$};
\node at (-3,1.15){\small $\om_{L}$};
\begin{scope}[shift={(-5.7,1.2)}]
\draw[->] (2,1.2)--(2.6,1.2);
\draw[->] (2.1,1.1)--(2.1,1.7);
\node at (2.65,1.3){\small $x$};
\node at (2.25,1.6){\small $y$};
\end{scope}
\node at (-3.5,1.5){\scriptsize Channel $1$};
\node at (0.5,1.5){\scriptsize Channel $2$};
\draw (-2.7,0.9)--(-2.7,1.1);
\draw (-0.3,0.9)--(-0.3,1.1);
\node at (-2.7,0.7){\small $-H$};
\node at (-0.3,0.7){\small $H$};
\end{tikzpicture}\qquad
\begin{tikzpicture}[scale=1.5]
\draw[fill=gray!30,draw=none](-2,1) rectangle (0.5,2);
\draw[fill=gray!30,draw=none](-3.5,1) rectangle (-2,2);
\draw[fill=gray!30,draw=none](-2,1.8) rectangle (-1,3);
\begin{scope}[scale=0.6,yshift=1.5cm,xshift=-3.2cm]
\draw [fill=white] plot [smooth cycle, tension=1] coordinates {(-0.6,0.9) (0,0.5) (0.7,1) (0.5,1.5) (-0.2,1.4)};
\end{scope}
\draw (-3.5,1)--(0.5,1); 
\draw (-3.5,2)--(-2,2)--(-2,3.2);
\draw (-1,3.2)--(-1,2)--(0.5,2); 
\draw[dashed] (-2,3.2)--(-2,3.4);
\draw[dashed] (-1,3.2)--(-1,3.4);
\draw[dashed] (1,1)--(0.5,1); 
\draw[dashed] (1,2)--(0.5,2);
\draw[dashed] (-4,2)--(-3.5,2);
\draw[dashed] (-4,1)--(-3.5,1);
\node at (-3,1.15){\small $\om_{\infty}$};
\node at (-3.5,1.5){\scriptsize Channel $1$};
\node at (0.5,1.5){\scriptsize Channel $2$};
\node at (-1.5,2.7){\scriptsize Channel $3$};
\phantom{\node at (-0.3,0.7){\small $H$};}
\end{tikzpicture}
\caption{Geometries of $\om_{L}$ (left) and $\om_{\infty}$ (right). \label{DomainOriginal2D}} 
\end{figure}

\noindent For $\ell>0$ and $L>1$, let $\om_L\subset\R^2$ be an open set which coincides with the domain
\begin{equation}\label{definitionWaveguide}
\{(x,y)\in\R\times (0;1)\cup (-\ell/2;\ell/2)\times [1;L)\}
\end{equation}
outside a bounded region independent of $\ell$, $L$. We assume that $\om_L$ is connected and has Lipschitz boundary $\partial\om_L$. We call channel 1 (resp. channel 2) the branch of $\om_L$ which coincides at infinity with $(-\infty;0)\times (0;1)$ (resp. with $(0;+\infty)\times (0;1)$) (see Figure \ref{DomainOriginal2D} left). We work in an academic geometry for the ease of presentation. In the conclusion (Section \ref{SectionConclusion}), we discuss possible extensions (see also Figure \ref{OtherGeom} right). We consider the problem with Neumann boundary condition
\begin{equation}\label{PbInitial}
\begin{array}{|rcll}
\Delta u + k^2 u & = & 0 & \mbox{ in }\om_L\\[3pt]
 \partial_nu  & = & 0  & \mbox{ on }\partial\om_L. 
\end{array}
\end{equation}
Here $\Delta$ is the $2\mrm{D}$ Laplace operator, $k$ corresponds to the wavenumber. Moreover, $n$ stands for the normal unit vector to $\partial\om_L$ directed to the exterior of $\om_L$. We take $k\in(0;\pi)$ so that
\[
w^{\pm}_{1}(x,y)=\cfrac{1}{\sqrt{2k}}\,e^{\mp i k x}\quad\mbox{ and }\quad  w^{\pm}_{2}(x,y)=\cfrac{1}{\sqrt{2k}}\,e^{\pm i k x}
\]
are the only propagating modes in channels $1$ and $2$. More precisely, for $i=1,2$, the planar mode $w^{+}_{i}$ (resp. $w^{-}_{i}$) is outgoing (resp. ingoing) in channel $i$. Introduce $\chi_l\in\mathscr{C}^{\infty}(\R^2)$ (resp. $\chi_r\in\mathscr{C}^{\infty}(\R^2)$) a cut-off function equal to one for $x\le -2H$ (resp. $x\ge 2H$) and to zero for $x\ge -H$ (resp. $x\le H$). Here $H$ is a parameter such that $\om_L$ coincides with $\R\times(0;1)$ for $|x|>H$. In order to describe the scattering process of the incident waves $w^{-}_{i}$ propagating in channel $i$, we introduce the following scattering solutions
\begin{equation}\label{DefScatteringCoeff}
\begin{array}{lcl}
u_{1}&=& \chi_l\,(w^{-}_{1}+s_{11}\,w^{+}_{1})+\chi_r\,s_{12}\,w^{+}_{2}+\tilde{u}_1\\[4pt]
u_{2}&=& \chi_l\,s_{21}\,w^{+}_{1}+\chi_r\,(w^{-}_{2}+s_{22}\,w^{+}_{2})+\tilde{u}_2,
\end{array}
\end{equation}
where $s_{ij}\in\Cplx$ and where $\tilde{u}_1$, $\tilde{u}_2$ decay exponentially as $O(e^{-\sqrt{\pi^2-k^2}|x|})$ for $x\to\pm\infty$. It is known (see e.g. \cite[Chap. 5, \S3.3, Thm. 3.5 p.160]{NaPl94}) that Problem (\ref{PbInitial}) admits solutions of the form (\ref{DefScatteringCoeff}). The complex constants $s_{ij}$, $i,j\in\{1,2\}$ in (\ref{DefScatteringCoeff}) are uniquely defined. Moreover the scattering matrix 
\begin{equation}\label{DefMatrixScattering}
\mathbb{S}:=\left(\begin{array}{cc}
s_{11} & s_{12}\\
s_{21} & s_{22}
\end{array}\right)\in\mathbb{C}^{2\times 2}
\end{equation}
is unitary ($\mathbb{S}\,\overline{\mathbb{S}}^{\top}=\mrm{Id}^{2\times2}$) and we have $s_{21}=s_{12}$, \textit{i.e.} $\mathbb{S}$ is symmetric (even for a non symmetric geometry). For the proof of the two latter properties, see Lemma \ref{LemmaUnitary} in Annex. Classically, $s_{11}$, $s_{22}$ are called \textit{reflection coefficients} while $s_{12}=s_{21}$ are \textit{transmission coefficients}.
In this section, for a given wavenumber $k\in(0;\pi)$, 
we explain how to construct geometries such that $s_{12}=0$.

\subsection{Asymptotic analysis of the scattering coefficients}\label{paragraphAsymptoAnalysis}

The scattering coefficients defined in (\ref{DefScatteringCoeff}), (\ref{DefMatrixScattering}) depend on the geometry, in particular on $L$. In this paragraph, we compute an asymptotic expansion of $s_{11}=s_{11}(L)$, $s_{12}=s_{21}(L)$, $s_{22}=s_{22}(L)$ as $L\to+\infty$. In the analysis, the properties of Problem (\ref{PbInitial}) set in the geometry $\om_{\infty}$ (see (\ref{DomainOriginal2D}) right) obtained from $\om_L$ making $L\to+\infty$, play a key role. We call channel 3 the branch of $\om_{\infty}$ which coincides at infinity with $(-\ell/2;\ell/2)\times(0;+\infty)$ (see Figure \ref{DomainOriginal2D} right). Channels 1 and 2 of $\om_{\infty}$ are defined as for $\om_L$. We assume that $\ell\in(0;\pi/k)$ so that  
\[
w^{\pm}_{3}(x,y)=\cfrac{1}{\sqrt{2k\ell}}\,e^{\pm i k y}
\]
are the only propagating modes in channel 3. For $\ell>\pi/k$, the analysis below must be adapted (see \cite[Section 5]{ChNPSu} for more details). In $\om_{\infty}$, there are the solutions 
\begin{equation}\label{defMatrixScaLim}
\begin{array}{lcl}
u_{1}^{\infty}&=& \chi_l(w^-_{1}+s^{\infty}_{11}\,w^+_{1})+\chi_r\,s^{\infty}_{12}\,w^+_{2}+\chi_t\,s^{\infty}_{13}\,w^+_{3}+\tilde{u}_{1}^{\infty}\\[4pt]
u_{2}^{\infty}&=&\chi_l\,s^{\infty}_{21}\,w^+_{1}+\chi_r\,(w^-_{2}+s^{\infty}_{22}\,w^+_{2})+\chi_t\,s^{\infty}_{23}\,w^+_{3}+\tilde{u}_{2}^{\infty}\\[4pt]
u_{3}^{\infty}&=&\chi_l\,s^{\infty}_{31}\,w^+_{1}+\chi_r\,s^{\infty}_{32}\,w^+_{2}+\chi_t\,(w^-_{3}+s^{\infty}_{33}\,w^+_{3})+\tilde{u}_{3}^{\infty},
\end{array}
\end{equation}
where $\tilde{u}_{1}^{\infty}$, $\tilde{u}_{2}^{\infty}$, $\tilde{u}_{3}^{\infty}$ decay exponentially at infinity. Here $\chi_t\in\mathscr{C}^{\infty}(\R^2)$ is a cut-off function equal to one for $y\ge 1+2\tau$ and to zero for $y\le 1+\tau$, where $\tau\ge0$ is a given constant such that $\om_{\infty}$ coincides with $(-\ell/2;\ell/2)\times\R$ for $y\ge\tau$. The scattering matrix
\begin{equation}\label{UnboundedScatteringMatrix}
\mathbb{S}^{\infty}:=\left(\begin{array}{ccc}
s^{\infty}_{11} & s^{\infty}_{12} & s^{\infty}_{13}\\
s^{\infty}_{21} & s^{\infty}_{22} & s^{\infty}_{23} \\
s^{\infty}_{31} & s^{\infty}_{32} & s^{\infty}_{33}
\end{array}\right)\in\mathbb{C}^{3\times 3}
\end{equation}
is uniquely defined, unitary ($\mathbb{S}^{\infty}\overline{\mathbb{S}^{\infty}}^{\top}=\mrm{Id}^{3\times3}$) and symmetric (work as in the proof of Lemma \ref{LemmaUnitary} in Annex to show the two latter properties). For $u_{1}$, $u_{2}$, following for example \cite[Chap. 5, \S5.6]{MaNP00}, we make the ansatzs
\begin{equation}\label{DefAnsatzs}
\begin{array}{lcl}
u_{1} &=& u_{1}^{\infty}+a_{1}(L)\,u_{3}^{\infty}+\dots \\
u_{2} &=& u_{2}^{\infty}+a_{2}(L)\,u_{3}^{\infty}+\dots \  
\end{array}
\end{equation}
where $a_{1}(L)$, $a_{2}(L)$ are some gauge functions to determine and where the dots correspond to small remainders. On $(-\ell/2;\ell/2)\times\{L\}$, the conditions $\partial_nu_{1}=0$, $\partial_nu_{2}=0$ lead to choose $a_{1}(L)$, $a_{2}(L)$ such that 
\[
s^{\infty}_{13}\,e^{ikL}+a_1(L)\,(-e^{-ikL}+s^{\infty}_{33}\,e^{ikL})=0\qquad\Leftrightarrow\qquad a_1(L)=\dsp\cfrac{s^{\infty}_{13}}{e^{-2ikL}-s^{\infty}_{33}}\phantom{\ .}
\]
\[
s^{\infty}_{23}\,e^{ikL}+a_2(L)\,(-e^{-ikL}+s^{\infty}_{33}\,e^{ikL})=0\qquad\Leftrightarrow\qquad a_2(L)=\dsp\cfrac{s^{\infty}_{23}}{e^{-2ikL}-s^{\infty}_{33}}\ .
\]
We shall consider ansatzs (\ref{DefAnsatzs}) when $|s^{\infty}_{33}|\ne1$ (when $|s^{\infty}_{33}|=1$, see \S\ref{paragraphDegenerated} in Annex). In this case, the gauge functions $a_1(L)$, $a_2(L)$ are well-defined for all $L>1$. Then we can prove that $\mathbb{S}=\mathbb{S}^{\mrm{asy}}(L)+\dots$, with
\begin{equation}\label{defMatriceLim}
\mathbb{S}^{\mrm{asy}}(L):=(s^{\mrm{asy}}_{ij})_{1\le i,j\le 2}=\left(\begin{array}{cc}
s^{\infty}_{11} & s^{\infty}_{12}\\
s^{\infty}_{21} & s^{\infty}_{22}
\end{array}\right)+\left(\begin{array}{c}a_1(L) \\ a_2(L)\end{array}\right)  \Big(\ s^{\infty}_{31}\quad  s^{\infty}_{32}\ \Big).
\end{equation}
In other words, we get
\begin{equation}\label{defCoeffLim}
\begin{array}{lcl}
s_{11} = s^{\infty}_{11}+\cfrac{s^{\infty}_{13}\,s^{\infty}_{31}}{e^{-2ikL}-s^{\infty}_{33}} +\dots&\qquad\ \,  &\qquad s_{12} = s^{\infty}_{12}+\cfrac{s^{\infty}_{13}\,s^{\infty}_{32}}{e^{-2ikL}-s^{\infty}_{33}} +\dots \\[12pt]
s_{21} = s^{\infty}_{21}+\cfrac{s^{\infty}_{23}\,s^{\infty}_{31}}{e^{-2ikL}-s^{\infty}_{33}} +\dots\ 
&\qquad \mbox{and} &\qquad s_{22} = s^{\infty}_{22}+\cfrac{s^{\infty}_{23}\,s^{\infty}_{32}}{e^{-2ikL}-s^{\infty}_{33}} +\dots.
\end{array}
\end{equation}
Here the dots stand for exponentially small terms. More precisely, we can establish (work as in \cite[Chap. 5, \S5.6]{MaNP00}, \cite[Prop. 8.1]{ChNPSu}) an error estimate of the form $\|\mathbb{S}-\mathbb{S}^{\mrm{asy}}(L)\| \le C\,e^{-\beta_{\ell} L}$ where $C$ is a constant independent of $L>1$ and $\beta_{\ell}:=\sqrt{(\pi/\ell)^2-k^2}$. Since $\mathbb{S}^{\infty}$ is symmetric, we see from (\ref{defCoeffLim}) that $\mathbb{S}^{\mrm{asy}}(L)$ is also symmetric. Moreover, after some calculations reproduced in the proof of Lemma \ref{lemmaUnitary} in Annex and based on the fact that $\mathbb{S}^{\infty}$ is unitary, one can check that for all $L>1$, the matrix $\mathbb{S}^{\mrm{asy}}(L)$ is unitary. Denote $\mathscr{C}:=\{z\in\Cplx\,|\,|z|=1\}$ the unit circle. As $L$ tends to $+\infty$, the coefficients $s^{\mrm{asy}}_{11}$, $s^{\mrm{asy}}_{12}=s^{\mrm{asy}}_{21}$ and $s^{\mrm{asy}}_{22}$ run respectively on the sets
\begin{equation}\label{setSPlusMoins}
\gamma_{11}:=\{s^{\infty}_{11}+\dsp\frac{s^{\infty}_{13}\,s^{\infty}_{31}}{z-s^{\infty}_{33}} \,|\,z\in\mathscr{C}\},\ \gamma_{12}:=\{s^{\infty}_{12}+\dsp\frac{s^{\infty}_{13}\,s^{\infty}_{32}}{z-s^{\infty}_{33}} \,|\,z\in\mathscr{C}\},\ \gamma_{22}:=\{s^{\infty}_{22}+\dsp\frac{s^{\infty}_{23}\,s^{\infty}_{32}}{z-s^{\infty}_{33}} \,|\,z\in\mathscr{C}\}.
\end{equation}
Using classical results concerning the M\"{o}bius transform (see e.g.  \cite[Chap. 5]{Henr74}), one finds that $\gamma_{11}$, $\gamma_{12}$, $\gamma_{22}$ are circles centered respectively at 
\begin{equation}\label{eqnCenters}
z_{11}:=s^{\infty}_{11}+\dsp\frac{s^{\infty}_{13}\,\overline{s^{\infty}_{33}}\,s^{\infty}_{31}}{1-|s^{\infty}_{33}|^{2}},\qquad z_{12}:=s^{\infty}_{12}+\dsp\frac{s^{\infty}_{13}\,\overline{s^{\infty}_{33}}\,s^{\infty}_{32}}{1-|s^{\infty}_{33}|^{2}},\qquad
z_{22}:=s^{\infty}_{22}+\dsp\frac{s^{\infty}_{23}\,\overline{s^{\infty}_{33}}\,s^{\infty}_{32}}{1-|s^{\infty}_{33}|^{2}}
\end{equation}
of radii 
\begin{equation}\label{eqnRadius}
\rho_{11}:=\dsp\frac{|s^{\infty}_{13}\,s^{\infty}_{31}|}{1-|s^{\infty}_{33}|^{2}},\qquad \rho_{12}:=\dsp\frac{|s^{\infty}_{13}\,s^{\infty}_{32}|}{1-|s^{\infty}_{33}|^{2}},\qquad \rho_{22}:=\dsp\frac{|s^{\infty}_{23}\,s^{\infty}_{32}|}{1-|s^{\infty}_{33}|^{2}}\ .
\end{equation}
In the following, we assume that the coefficients $s_{12}^{\infty}$, $s_{13}^{\infty}$, $s_{23}^{\infty}$ in (\ref{defMatrixScaLim}) are such that $s_{12}^{\infty}\,s_{13}^{\infty}\,s_{23}^{\infty}\ne0$. This assumption is needed so that couplings exist between the three channels of $\om_{\infty}$. We refer the reader to \S\ref{paragraphPerfectRef} for an example of situation where numerically this assumption is satisfied. When $s_{12}^{\infty}\,s_{13}^{\infty}\,s_{23}^{\infty}\ne0$, as a consequence of the unitary structure of $\mathbb{S}^{\infty}$, we have $|s_{33}^{\infty}|\ne1$ and the three radii $\rho_{11}$, $\rho_{12}$, $\rho_{22}$ are located in $(0;1)$.  
\begin{proposition}\label{PropositionGoesThroughZero}
Assume that the coefficients $s_{12}^{\infty}$, $s_{13}^{\infty}$, $s_{23}^{\infty}$ in (\ref{defMatrixScaLim}) satisfy $s_{12}^{\infty}\,s_{13}^{\infty}\,s_{23}^{\infty}\ne0$. Then the circle $\gamma_{12}$ (the asymptotic orbit for the transmission coefficient) passes through zero. 
\end{proposition}
\begin{proof}
We have 
\[
s^{\infty}_{12}+\dsp\frac{s^{\infty}_{13}\,s^{\infty}_{32}}{z-s^{\infty}_{33}}=0\qquad\Leftrightarrow\qquad z=s^{\infty}_{33}-\dsp\frac{s^{\infty}_{13}\,s^{\infty}_{32}}{s^{\infty}_{12}}.
\]
Therefore, it is sufficient to prove that 
\begin{equation}\label{ConditionToVerifyPart1}
|s^{\infty}_{33}-\dsp\frac{s^{\infty}_{13}\,s^{\infty}_{32}}{s^{\infty}_{12}}|=1.
\end{equation}
Set $C=|s^{\infty}_{33}|^2|s^{\infty}_{12}|^2+
|s^{\infty}_{13}|^2|s^{\infty}_{32}|^2
-2\,\Re e\,(s^{\infty}_{13}\,\overline{s^{\infty}_{33}}\,s^{\infty}_{32}\overline{s^{\infty}_{12}})
-|s^{\infty}_{12}|^2$. A direct computation shows that (\ref{ConditionToVerifyPart1}) is true if and only if $C=0$. Since $\mathbb{S}^{\infty}$ is unitary, we have 
\[
s^{\infty}_{11}\overline{s^{\infty}_{31}}+
s^{\infty}_{12}\overline{s^{\infty}_{32}}+
s^{\infty}_{13}\overline{s^{\infty}_{33}}=0\qquad\Leftrightarrow \qquad 
s^{\infty}_{12}\overline{s^{\infty}_{32}}+
s^{\infty}_{13}\overline{s^{\infty}_{33}}=-s^{\infty}_{11}\overline{s^{\infty}_{31}}.
\]
Squaring the second equality above, we deduce that 
\[
2\,\Re e\,(s^{\infty}_{13}\,\overline{s^{\infty}_{33}}\,s^{\infty}_{32}\overline{s^{\infty}_{12}})=|s^{\infty}_{11}|^2|s^{\infty}_{31}|^2-|s^{\infty}_{12}|^2|s^{\infty}_{32}|^2-|s^{\infty}_{13}|^2|s^{\infty}_{33}|^2.
\]
Using again the fact that $\mathbb{S}^{\infty}$ is unitary, we can write
\[
\begin{array}{lcl}
C  &= & |s^{\infty}_{33}|^2(|s^{\infty}_{12}|^2+|s^{\infty}_{13}|^2)+|s^{\infty}_{32}|^2(|s^{\infty}_{12}|^2+|s^{\infty}_{13}|^2)-|s^{\infty}_{11}|^2|s^{\infty}_{31}|^2-|s^{\infty}_{12}|^2\\[4pt]
&= & |s^{\infty}_{33}|^2(1-|s^{\infty}_{11}|^2)+|s^{\infty}_{32}|^2(1-|s^{\infty}_{11}|^2)-|s^{\infty}_{11}|^2|s^{\infty}_{31}|^2-|s^{\infty}_{12}|^2\\[4pt]
&= & 1-|s^{\infty}_{31}|^2-|s^{\infty}_{11}|^2-|s^{\infty}_{12}|^2\ =\ 0.
\end{array}
\]
This gives the desired result.
\end{proof}

\noindent Proposition \ref{PropositionGoesThroughZero} together with the error estimate $\|\mathbb{S}-\mathbb{S}^{\mrm{asy}}(L)\| \le C\,e^{-\beta_{\ell} L}$ show that the curve $L\mapsto s_{12}(L)$ for the transmission coefficient goes as close as we wish to zero as $L\to+\infty$. In the next paragraph, we prove that due  to the unitary structure of $\mathbb{S}$, the curve $L\mapsto s_{12}(L)$ passes exactly through zero. 

\subsection{Zero transmission}

Now we state and prove the main result of this section.
\begin{theorem}\label{MainThmPart1}
Assume that the coefficients $s_{12}^{\infty}$, $s_{13}^{\infty}$, $s_{23}^{\infty}$ in (\ref{defMatrixScaLim}) satisfy $s_{12}^{\infty}\,s_{13}^{\infty}\,s_{23}^{\infty}\ne0$. Then the complex curve $L\mapsto s_{12}(L)$ for the transmission coefficient passes through zero an infinite number of times as $L\to+\infty$.
\end{theorem}
\begin{remark}
In Section \ref{SectionNumExpe}, \S\ref{paragraphPerfectRef} below, we provide an example of situation where numerically the assumption of Theorem \ref{MainThmPart1} is satisfied. 
\end{remark}
\begin{proof}
\begin{figure}[!ht]
\centering
\begin{tikzpicture}[scale=3]
\draw[draw=none,fill=red!30] (0,0)--(8.13:1)--(8.13:1)arc(8.13:98.13:1)--cycle;
\draw[draw=none,fill=red!30] (0,0)--(188.13:1)--(188.13:1)arc(188.13:278.13:1)--cycle;
\draw[blue,line width=0.3mm] (4/5,-3/5) circle (5/5);
\draw[green!60!black,line width=0.3mm] (4/5,-3/5+0.1) ellipse (1.1 and 0.9);
\draw[->] (-1.2,0) -- (1.2,0);
\draw[->] (0,-1.2) -- (0,1.2);
\draw[red,line width=0.3mm] (0,0)--(8.13:1);
\draw[red,line width=0.3mm] (0,0)--(98.13:1);
\draw[red,line width=0.3mm] (0,0)--(188.13:1);
\draw[red,line width=0.3mm] (0,0)--(278.13:1);
\draw[domain=-0.8:0.8,smooth,variable=\x] plot ({\x},{4*\x/3.});
\node[blue] at (1.4,-1.55){$\gamma_{12}$};
\node[magenta] at (-0.26,0.17){$s_{12}(a_n)$};
\node[magenta] at (-0.3,-0.15){$s_{12}(b_n)$};
\node[green!60!black] at (2.2,-0.2){$L\mapsto s_{12}(L)$};
\node at (-0,0.8){$Q_{1}$};
\node at (0,-0.8){$Q_{2}$};
\fill[magenta] (-0.015,0.11) circle (0.02);
\fill[magenta] (-0.13,-0.02) circle (0.02);
\node at (1.4,1){$\{\rho\,e^{i\eta}\in\Cplx,\ \rho\in\R\}$};
\end{tikzpicture}
\caption{Notation used in the proof of Theorem \ref{MainThmPart1}.\label{PictureProof}}
\end{figure}
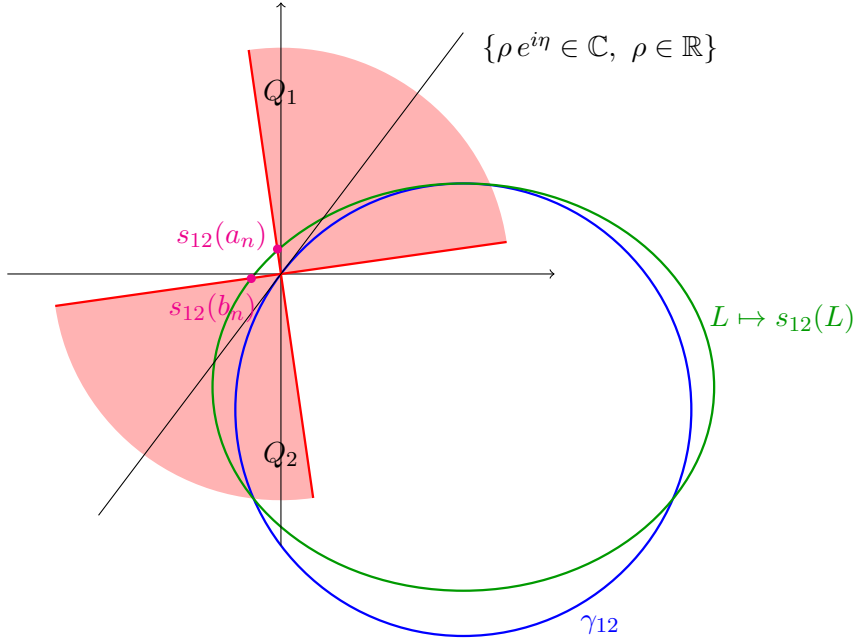

According to Proposition \ref{PropositionGoesThroughZero}, we know that there is $L_{\star}>0$ such that $s_{12}^{\mrm{asy}}(L_{\star})=0$. 
Since $L\mapsto s_{12}^{\mrm{asy}}(L)$ is $\pi/k$-periodic (see (\ref{defMatriceLim}), (\ref{defCoeffLim})), we deduce that $s_{12}^{\mrm{asy}}(L_{\star}+n\pi/k)=0$ for all $n\in\N:=\{0,1,\dots\}$. Define the interval
\[
I_n := (\,L_{\star}+\cfrac{n\pi}{k}-\cfrac{1}{n+1}\ ;\ L_{\star}+\cfrac{n\pi}{k}+\cfrac{1}{n+1}\,)\ .
\] 
Since the circle $\gamma_{12}$ passes through zero, there is $\eta\in(-\pi/2;\pi/2]$ such that $\gamma_{12}$ is tangent to the line $\{\rho\,e^{i\eta}\in\Cplx,\ \rho\in\R\}$. Define the quadrants 
\[
\begin{array}{ll}
Q_1:=\{\rho\,e^{i\theta}\in\Cplx\,|\,\rho>0,\ \eta-\pi/4<\theta<\eta+\pi/4\}\\[2pt]
Q_2:=\{\rho\,e^{i\theta}\in\Cplx\,|\,\rho<0,\ \eta-\pi/4<\theta<\eta+\pi/4\},
\end{array}
\]
see Figure \ref{PictureProof}. On $I_n$, the curve $L\mapsto s_{12}^{\mrm{asy}}(L)$ meets both quadrants $Q_1$ and $Q_2$. On the other hand, we have $|s_{12}(L)-s_{12}^{\mrm{asy}}(L)|\le C\,e^{-\beta_{\ell} L}$. As a consequence, there is $N\in\N$ such that for all $n\ge N$, the maps  $L\mapsto s_{12}(L)$ intersects both $Q_1$ and $Q_2$ on $I_n$.\\
\newline
Assume by contradiction that for all $n\ge N$, $L\mapsto s_{12}(L)$ does not pass through zero on $I_n$. Let us first describe the idea that we will use before making it rigorous. Since $\mathbb{S}$ is unitary, there holds
\begin{equation}\label{RelationUnitary}
s_{11}(L)\,\overline{s_{12}(L)}+s_{12}(L)\,\overline{s_{22}(L)}=0.
\end{equation}
As a consequence, we can write $s_{12}(L)/\overline{s_{12}(L)}=-s_{11}(L)/\overline{s_{22}(L)}$ for all $L\in I_n$. But if $L\mapsto s_{12}(L)$ does not pass through zero on $I_n$, the point $s_{12}(L)/\overline{s_{12}(L)}=e^{2i\mrm{arg}(s_{12}(L))}$ must run rapidly on the unit circle for $L\in I_n$ as $n\to+\infty$. On the other hand, $s_{11}(L)/\overline{s_{22}(L)}$ tends to a constant on $I_n$ as $n\to+\infty$. This way we obtain a contradiction. We emphasize that the unitary structure of $\mathbb{S}$ is the key ingredient of the proof.\\
\newline
Now we make this discussion rigorous. If  $L\mapsto s_{12}(L)$ does not vanish on $I_n$, since $L\mapsto s_{12}(L)$ is continuous and intersects both $Q_1$, $Q_2$ on $I_n$, we deduce that for all $n\ge N$, there are $a_n,\,b_n\in I_n$ such that $s_{12}(a_n)=t_n\,e^{i(\eta-\pi/4)}$ and $s_{12}(b_n)=\tilde{t}_n\,e^{i(\eta+\pi/4)}$, with $t_n$, $\tilde{t}_n\in\R\setminus\{0\}$. Taking successively $L=a_n$, $L=b_n$ in (\ref{RelationUnitary}), we obtain 
\begin{equation}\label{RelationContradicton}
s_{11}(a_n)=-ie^{2i\eta}\overline{s_{22}(a_n)}\qquad\mbox{ and }\qquad s_{11}(b_n)=ie^{2i\eta}\overline{s_{22}(b_n)}. 
\end{equation}
Pick $\eps$ small enough so that the open balls\footnote{For $z_0\in\Cplx$ (resp. $z_0\in\R^2$), $B(z_0,r)$ denotes the open ball of $\Cplx$ (resp. $\R^2$) of radius $r>0$ centered at $z_0$.} $B(ie^{2i\eta}\overline{s_{22}^{\mrm{asy}}(L_{\star}}),\eps)$ and $B(-ie^{2i\eta}\overline{s_{22}^{\mrm{asy}}(L_{\star})},\eps)$ do not intersect. This is possible because $|s_{22}^{\mrm{asy}}(L_{\star})|=1$ (remember that $s_{12}^{\mrm{asy}}(L_{\star})=0$). We know that there is $M$ large enough such that, for all $n\ge M$, we have $s_{11}(a_n),\,s_{11}(b_n)\in B(s_{11}^{\mrm{asy}}(L_{\star}),\eps/2)$ and $s_{22}(a_n),\,s_{22}(b_n)\in B(s_{22}^{\mrm{asy}}(L_{\star}),\eps/2)$. From (\ref{RelationContradicton}), we deduce that we must have both
\[
B(s_{11}^{\mrm{asy}}(L_{\star}),\eps/2)\cap B(ie^{2i\eta}\overline{s_{22}^{\mrm{asy}}(L_{\star})},\eps/2)\neq\emptyset\ \mbox{ and }\  B(s_{11}^{\mrm{asy}}(L_{\star}),\eps/2)\cap B(-ie^{2i\eta}\overline{s_{22}^{\mrm{asy}}(L_{\star})},\eps/2)\neq\emptyset.
\]
This is impossible because there holds $B(\overline{s_{22}^{\mrm{asy}}(L_{\star}}),\eps)\cap B(-\overline{s_{22}^{\mrm{asy}}(L_{\star})},\eps)=\emptyset$. This shows that for all $n\ge\max(M,N)$, $L\mapsto s_{12}(L)$ cancels on $I_n$. As a consequence, $L\mapsto s_{12}(L)$ passes through zero an infinite number of times as $L\to+\infty$. More precisely, it passes trough zero almost periodically with a period $\pi/k$.
\end{proof}

\section{Trapped modes}\label{SectionTrappedModes}

In the previous section, for a given wavenumber $k_0\in(0;\pi)$, we showed how to construct geometries with two open channels such that the transmission coefficient $s_{12}$ defined in (\ref{DefScatteringCoeff}) is equal to zero. Now, from a given waveguide $\Om_{\infty}$ where we know that $s_{12}=0$, truncating one open channel, we explain how to find a geometry $\Om_{\mathcal{L}}$ with one open channel supporting trapped modes for Problem (\ref{PbInitial}) at wavenumbers $k$ close to $k_0$. We remind the reader that we say that $u$ is a trapped mode for Problem (\ref{PbInitial}) if $u$ belongs to the Sobolev space $\mH^1(\Om_{\mathcal{L}})$ and verifies (\ref{PbInitial}). Note that $\Om_{\infty}$, the waveguide where $s_{12}=0$, can result from the construction of the previous section ($\Om_{\infty}=\om_L$), but not necessarily.

\subsection{Augmented scattering matrix}

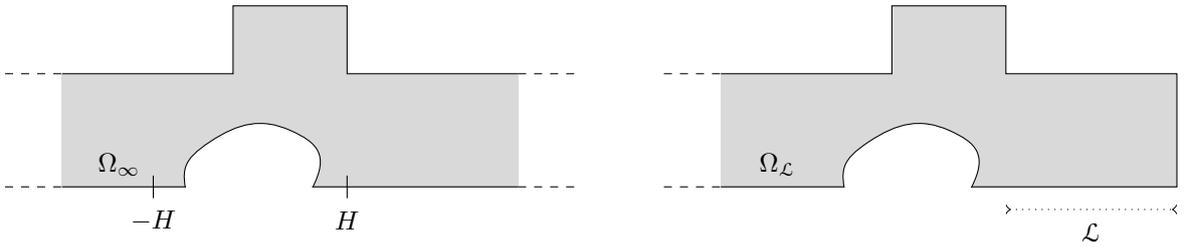
\begin{figure}[!ht]
\centering
\begin{tikzpicture}[scale=1.5]
\draw[fill=gray!30,draw=none](-2,1) rectangle (0.5,2);
\draw[fill=gray!30,draw=none](-3.5,1) rectangle (-2,2);
\draw[fill=gray!30,draw=none](-2,1.8) rectangle (-1,2.6);
\draw (-2,1)--(0.5,1); 
\draw (-3.5,2)--(-2,2)--(-2,2.6)--(-1,2.6)--(-1,2)--(0.5,2); 
\draw (-3.5,1)--(-2,1); 
\draw (-2.7,0.9)--(-2.7,1.1); 
\draw (-1,0.9)--(-1,1.1); 
\draw[dashed] (0.5,1)--(1,1); 
\draw[dashed] (0.5,2)--(1,2);
\draw[dashed] (-4,2)--(-3.5,2);
\draw[dashed] (-4,1)--(-3.5,1);
\node at (-3,1.2){\small $\Om_{\infty}$};
\node at (-2.7,0.7){\small $-H$};
\node at (-1,0.7){\small $H$};
\begin{scope}
\clip (-3,0.9) rectangle (0.5,3);
\draw [fill=white,draw=none] plot [smooth cycle, tension=1] coordinates {(-2.4,0.9) (-2,0.5) (-1.3,1) (-1.5,1.5) (-2.2,1.4)};
\end{scope}
\begin{scope}
\clip (-3,1) rectangle (0.5,3);
\draw [fill=white] plot [smooth cycle, tension=1] coordinates {(-2.4,0.9) (-2,0.5) (-1.3,1) (-1.5,1.5) (-2.2,1.4)};
\end{scope}
\phantom{\node at (-0.25,0.6){\small $\mathcal{L}$};}
\end{tikzpicture}\qquad\quad\begin{tikzpicture}[scale=1.5]
\draw[fill=gray!30,draw=none](-2,1) rectangle (0.5,2);
\draw[fill=gray!30,draw=none](-3.5,1) rectangle (-2,2);
\draw[fill=gray!30,draw=none](-2,1.8) rectangle (-1,2.6);
\draw (-3.5,2)--(-2,2)--(-2,2.6)--(-1,2.6)--(-1,2)--(0.5,2)--(0.5,1)--(-2,1); 
\draw (-3.5,1)--(-2,1); 
\draw[dashed] (-4,2)--(-3.5,2);
\draw[dashed] (-4,1)--(-3.5,1);
\draw[dotted,>-<] (-1,0.8)--(0.5,0.8);
\node at (-0.25,0.6){\small $\mathcal{L}$};
\node at (-3,1.2){\small $\Om_{\mathcal{L}}$};
\begin{scope}
\clip (-3,0.9) rectangle (0.5,3);
\draw [fill=white,draw=none] plot [smooth cycle, tension=1] coordinates {(-2.4,0.9) (-2,0.5) (-1.3,1) (-1.5,1.5) (-2.2,1.4)};
\end{scope}
\begin{scope}
\clip (-3,1) rectangle (0.5,3);
\draw [fill=white] plot [smooth cycle, tension=1] coordinates {(-2.4,0.9) (-2,0.5) (-1.3,1) (-1.5,1.5) (-2.2,1.4)};
\end{scope}
\end{tikzpicture}
\caption{Geometries of $\Om_{\infty}$ (left) and $\Om_{\mathcal{L}}$ (right). \label{DomainTrapped}} 
\end{figure}

\noindent Let us start with a waveguide $\Om_{\infty}\subset\R^2$ which coincides for $\pm x > H$, where $H>0$ is given, with the strip $\R\times(0;1)$ (see Figure \ref{DomainTrapped} left). Again, this geometry is chosen only to simplify the presentation and other settings where the analysis below can be conducted are discussed in Section \ref{SectionConclusion}. For a given wavenumber $k\in(0;\pi)$, we assume that $\Om_{\infty}$ is such that the transmission coefficient $s_{12}$ appearing in (\ref{DefScatteringCoeff}) is equal to zero (zero transmission). We refer the reader to Section \ref{SectionPerfectReflectivity} or to \cite{ChNPSu,ChPaSu} for the construction of such domains. Then for $\mathcal{L}>0$, we define the half-waveguide (unbounded in the left direction)
\[
\Om_\mathcal{L}:=\{(x,y)\in\Om_{\infty}\,|\,x<H+\mathcal{L}\}.
\]
In the following, we explain how to find $\mathcal{L}>0$ such that trapped modes exist for Problem (\ref{PbInitial}) in $\Om_\mathcal{L}$. We impose  Neumann boundary conditions on $\{H+\mathcal{L}\}\times(0;1)$ but we could also work with Dirichlet. \\ 
\newline
Set $\beta:=\sqrt{\pi^2-k^2}$ and 
\[
W^{\pm}_{1}(x,y)=w^{\pm}_{1}(x,y)=\cfrac{1}{\sqrt{2k}}\,e^{\mp i k x},\qquad W^{\pm}_{2}(x,y)=\cfrac{1}{\sqrt{2\beta}}\,(e^{-\beta x}\mp i e^{\beta x })\cos(\pi y).
\]
Note the particular definition of the functions $W^{\pm}_{2}$ which are ``wave packets'', combinations of exponentially decaying and growing modes as $x\to-\infty$. The normalisation coefficients for $W^{\pm}_{1}$, $W^{\pm}_{2}$ are chosen so that the matrix defined in (\ref{DefAugmentedMatrixScattering}) is unitary. In \cite{NaPl94bis,KaNa02,Naza06,Naza11}, it is proved that in the half-waveguide $\Om_\mathcal{L}$, there are the solutions
\begin{equation}\label{DefSolutionScaAug}
\begin{array}{lcl}
U_{1}&=&W^-_{1}+S_{11}\,W^+_{1}+S_{12}\,W^+_{2}+\tilde{U}_{1}\\[4pt]
U_{2}&=&W^-_{2}+S_{21}\,W^+_{1}+S_{22}\,W^+_{2}+\tilde{U}_{2}
\end{array}
\end{equation}
where $\tilde{U}_{1}$, $\tilde{U}_{2}$ decay as $O(e^{\sqrt{4\pi^2-k^2}x})$ when $x\to-\infty$. The complex constants $S_{ij}$, $i,j\in\{1,2\}$ in (\ref{DefSolutionScaAug}) are uniquely defined. They allow us to define the augmented scattering matrix introduced in \cite{NaPl94bis,KaNa02,Naza06,Naza11}
\begin{equation}\label{DefAugmentedMatrixScattering}
\mathcal{S}:=\left(\begin{array}{cc}
S_{11} & S_{12}\\
S_{21} & S_{22}
\end{array}\right)\in\mathbb{C}^{2\times 2}.
\end{equation}
Working exactly as in the proof of Lemma \ref{LemmaUnitary} in Annex, one can show that $\mathcal{S}$ is unitary ($\mathcal{S}\,\overline{\mathcal{S}}^{\top}=\mrm{Id}^{2\times2}$) and symmetric ($S_{21}=S_{12}$). This augmented scattering matrix provides an algebraic criterion to detect the presence of trapped modes (see e.g. \cite[Thm. 2]{Naza11}):
\begin{lemma}\label{LemmaExistenceTrappedMode}
If $S_{22}=-1$, then $U_{2}$ is a trapped mode for Problem (\ref{PbInitial}) set in $\Om_\mathcal{L}$. 
\end{lemma}
\begin{remark}
Note that $S_{22}=-1$ is only a sufficient criterion of existence of trapped modes. Indeed the geometry $\Om_\mathcal{L}$ can support trapped modes for Problem (\ref{PbInitial}) with $S_{22}\ne-1$. In this case, these trapped modes must decay as $O(e^{\sqrt{4\pi^2-k^2}x})$ when $x\to-\infty$.
\end{remark}
\begin{proof}
If $S_{22}=-1$, since $\mathcal{S}$ is unitary, then $S_{21}=0$. In such a situation, according to (\ref{DefSolutionScaAug}), we have $U_2=-i\sqrt{2/\beta}\,e^{\beta x}\cos(\pi y)+O(e^{\sqrt{4\pi^2-k^2}x})$ as $x\to-\infty$. This shows that $U_2\not\equiv0$ belongs to $\mH^1(\Om_\mathcal{L})$. In other words $U_2$ is a trapped mode. 
\end{proof}

\subsection{Asymptotic analysis of the coefficients of the augmented scattering matrix}\label{paragraphAsymptoCoefAug}
The augmented scattering matrix $\mathcal{S}$ depends on $\mathcal{L}$. In this section, we compute an asymptotic expansion of all the elements of $\mathcal{S}(\mathcal{L})$ as $\mathcal{L}\to+\infty$. In particular, we will show that the dominant asymptotic term of $\mathcal{L}\mapsto S_{22}(\mathcal{L})$ hits the unit circle.\\
To proceed, we work exactly as in \S\ref{paragraphAsymptoAnalysis}. In the geometry $\Om_{\infty}$ (see Figure \ref{DomainTrapped} right) obtained from $\Om_{\mathcal{L}}$ making $\mathcal{L}\to+\infty$, there are the solutions 
\[
\begin{array}{lcl}
U_{1}^{\infty}&=& \chi_l(W^-_{1}+S^{\infty}_{11}\,W^+_{1}+S^{\infty}_{12}\,W^+_{2})+\chi_r\,S^{\infty}_{13}\,W^+_{1}+\tilde{U}_{1}^{\infty}\\[4pt]
U_{2}^{\infty}&=&\chi_l(W^-_{2}+S^{\infty}_{21}\,W^+_{1}+S^{\infty}_{22}\,W^+_{2})+\chi_r\,S^{\infty}_{23}\,W^+_{1}+\tilde{U}_{2}^{\infty}\\[4pt]
U_{3}^{\infty}&=&\chi_l(S^{\infty}_{31}\,W^+_{1}+S^{\infty}_{32}\,W^+_{2})+\chi_r\,(W^-_{1}+S^{\infty}_{33}\,W^+_{1})+\tilde{U}_{3}^{\infty},
\end{array}
\]
where $\tilde{U}_{1}^{\infty}$, $\tilde{U}_{2}^{\infty}$, $\tilde{U}_{3}^{\infty}$ decay as $O(e^{\sqrt{4\pi^2-k^2}x})$ for $x\to-\infty$ and as $O(e^{-\sqrt{\pi^2-k^2}x})$ for $x\to+\infty$. Here $\chi_l$, $\chi_r$ are the cut-off functions introduced in (\ref{DefScatteringCoeff}). The scattering matrix
\begin{equation}\label{MatriceScaAugLim}
\mathcal{S}^{\infty}:=\left(\begin{array}{ccc}
S^{\infty}_{11} & S^{\infty}_{12} & S^{\infty}_{13}\\
S^{\infty}_{21} & S^{\infty}_{22} & S^{\infty}_{23} \\
S^{\infty}_{31} & S^{\infty}_{32} & S^{\infty}_{33}
\end{array}\right)\in\mathbb{C}^{3\times 3}
\end{equation}
is unitary and symmetric. For $U_{1}$, $U_{2}$, we make the ansatzs
\[
\begin{array}{lcl}
U_{1} &=& U_{1}^{\infty}+A_{1}(\mathcal{L})\,U_{3}^{\infty}+\dots \\
U_{2} &=& U_{2}^{\infty}+A_{2}(\mathcal{L})\,U_{3}^{\infty}+\dots \  .
\end{array}
\]
On $\{H+L\}\times(0;1)$, the conditions $\partial_nU_{1}=0$, $\partial_nU_{2}=0$ lead to choose $A_{1}(\mathcal{L})$, $A_{2}(\mathcal{L})$ such that 
\[
S^{\infty}_{13}\,e^{ik\mathcal{L}}+A_1(\mathcal{L})\,(-e^{-ik\mathcal{L}}+S^{\infty}_{33}\,e^{ik\mathcal{L}})=0\qquad\Leftrightarrow\qquad A_1(\mathcal{L})=\dsp\cfrac{S^{\infty}_{13}}{e^{-2ik\mathcal{L}}-S^{\infty}_{33}}
\]
\[
S^{\infty}_{23}\,e^{ik\mathcal{L}}+A_2(\mathcal{L})\,(-e^{-ik\mathcal{L}}+S^{\infty}_{33}\,e^{ik\mathcal{L}})=0\qquad\Leftrightarrow\qquad A_2(\mathcal{L})=\dsp\cfrac{S^{\infty}_{23}}{e^{-2ik\mathcal{L}}-S^{\infty}_{33}}\,.
\]
As in \S\ref{paragraphAsymptoAnalysis}, we assume that $|S^{\infty}_{33}|\ne1$ (see \S\ref{paragraphDegenerated} in the Annex for the case $|S^{\infty}_{33}|=1$). Then we have $\mathcal{S}=\mathcal{S}^{\mrm{asy}}(\mathcal{L})+\dots$, with
\[
\mathcal{S}^{\mrm{asy}}(\mathcal{L})=(\mathcal{S}^{\mrm{asy}}_{ij})_{1\le i,j\le 2}:=\left(\begin{array}{cc}
S^{\infty}_{11} & S^{\infty}_{12}\\
S^{\infty}_{21} & S^{\infty}_{22}
\end{array}\right)+\left(\begin{array}{c}A_1(\mathcal{L}) \\ A_2(\mathcal{L})\end{array}\right)  \Big(\ S^{\infty}_{31}\quad  S^{\infty}_{32}\ \Big).
\]
In other words, we have
\begin{equation}\label{AsymptoticAugm}
\begin{array}{lcl}
S_{11} = S^{\infty}_{11}+\cfrac{S^{\infty}_{13}\,S^{\infty}_{31}}{e^{-2ik\mathcal{L}}-S^{\infty}_{33}} +\dots&\qquad \qquad&\qquad S_{12} = S^{\infty}_{12}+\cfrac{S^{\infty}_{13}\,S^{\infty}_{32}}{e^{-2ik\mathcal{L}}-S^{\infty}_{33}} +\dots\quad\\[10pt]
S_{21} = S^{\infty}_{21}+\cfrac{S^{\infty}_{23}\,S^{\infty}_{31}}{e^{-2ik\mathcal{L}}-S^{\infty}_{33}} +\dots\ 
&\qquad \mbox{and}&\qquad S_{22} = S^{\infty}_{22}+\cfrac{S^{\infty}_{23}\,S^{\infty}_{32}}{e^{-2ik\mathcal{L}}-S^{\infty}_{33}} +\dots.
\end{array}
\end{equation}
The dots stand for exponentially small terms. More precisely, we can establish an error estimate of the form $\|\mathcal{S}-\mathcal{S}^{\mrm{asy}}(\mathcal{L})\| \le C\,e^{-\beta \mathcal{L}}$ where $C$ is a constant independent of $\mathcal{L}>0$. Since $\mathcal{S}^{\infty}$ is symmetric, $\mathcal{S}^{\mrm{asy}}(\mathcal{L})$ is also symmetric. Moreover working as in Lemma \ref{lemmaUnitary} in Annex, one can check that for all $\mathcal{L}>0$, $\mathcal{S}^{\mrm{asy}}(\mathcal{L})$ is unitary. As $\mathcal{L}$ tends to $+\infty$, the coefficients $S_{11}$, $S_{12}=S_{21}$ and $S_{22}$ run respectively on the sets
\begin{equation}\label{setSPlusMoins}
\Gamma_{11}:=\{S^{\infty}_{11}+\dsp\frac{S^{\infty}_{13}\,S^{\infty}_{31}}{z-S^{\infty}_{33}} \,|\,z\in\mathscr{C}\},\ \Gamma_{12}:=\{S^{\infty}_{12}+\dsp\frac{S^{\infty}_{13}\,S^{\infty}_{32}}{z-S^{\infty}_{33}} \,|\,z\in\mathscr{C}\},\ \Gamma_{22}:=\{S^{\infty}_{22}+\dsp\frac{S^{\infty}_{23}\,S^{\infty}_{32}}{z-S^{\infty}_{33}} \,|\,z\in\mathscr{C}\}.
\end{equation}
As in the previous section, one finds that $\Gamma_{11}$, $\Gamma_{12}$, $\Gamma_{22}$ coincide with circles centered respectively at 
\begin{equation}\label{eqnCenters}
Z_{11}:=S^{\infty}_{11}+\dsp\frac{S^{\infty}_{13}\,\overline{S^{\infty}_{33}}\,S^{\infty}_{31}}{1-|S^{\infty}_{33}|^{2}},\qquad Z_{12}:=S^{\infty}_{12}+\dsp\frac{S^{\infty}_{13}\,\overline{S^{\infty}_{33}}\,S^{\infty}_{32}}{1-|S^{\infty}_{33}|^{2}},\qquad
Z_{22}:=S^{\infty}_{22}+\dsp\frac{S^{\infty}_{23}\,\overline{S^{\infty}_{33}}\,S^{\infty}_{32}}{1-|S^{\infty}_{33}|^{2}}
\end{equation}
of radii
\begin{equation}\label{eqnRadius}
P_{11}:=\dsp\frac{|S^{\infty}_{13}\,S^{\infty}_{31}|}{1-|S^{\infty}_{33}|^{2}},\qquad P_{12}:=\dsp\frac{|S^{\infty}_{13}\,S^{\infty}_{32}|}{1-|S^{\infty}_{33}|^{2}},\qquad P_{22}:=\dsp\frac{|S^{\infty}_{23}\,S^{\infty}_{32}|}{1-|S^{\infty}_{33}|^{2}}.
\end{equation}
Working exactly as in the proof of Proposition \ref{PropositionGoesThroughZero}, we can show the following statement.
\begin{proposition}\label{PropositionAugmScaUnit}
Assume that the coefficients $S_{12}^{\infty}$, $S_{13}^{\infty}$, $S_{23}^{\infty}$ in (\ref{MatriceScaAugLim}) satisfy $S_{12}^{\infty}\,S_{13}^{\infty}\,S_{23}^{\infty}\ne0$. Then the circle $\Gamma_{12}$ passes through zero. As a consequence, since $\mathcal{S}^{\mrm{asy}}(\mathcal{L})$ is unitary, we deduce that the circles $\Gamma_{11}$, $\Gamma_{22}$ intersect the unit circle $\mathscr{C}$ at exactly one point.
\end{proposition}
\begin{remark}
From this proposition, it follows that the dominant asymptotic term of $\mathcal{L}\mapsto S_{22}(\mathcal{L})$ hits the unit circle. This is the property that will be used below to prove the existence of trapped modes. 
\end{remark}
\begin{remark}
As in the previous section, we assume that $S_{12}^{\infty}\,S_{13}^{\infty}\,S_{23}^{\infty}\ne0$, which implies $|S_{33}^{\infty}|\ne1$, so that couplings exist between the modes in $\Om_{\infty}$. 
\end{remark}

\subsection{The particular case where zero transmission occurs in $\Om_{\infty}$}\label{paragraphParticularCase}
In $\Om_{\infty}$, there are also the classical solutions $u_1$, $u_2$ introduced in (\ref{DefScatteringCoeff}) which allow one to define the usual scattering matrix $\mathbb{S}\in\Cplx^{2\times 2}$ in (\ref{DefMatrixScattering}) (for the identity relating $\mathbb{S}\in\Cplx^{2\times 2}$ and $\mathcal{S}^{\infty}\in\Cplx^{3\times 3}$, we refer the reader to \cite[Thm. 3]{Naza11}). In this section, we are interested in situations (geometries) where $s_{12}=s_{21}=0$ (zero transmission). 
In this case, we establish an additional property for the asymptotic circle $\Gamma_{22}$ defined in (\ref{setSPlusMoins}).
\begin{proposition}\label{PropositionThroughMinusOne}
Assume that the coefficients $S_{12}^{\infty}$, $S_{13}^{\infty}$, $S_{23}^{\infty}$ in (\ref{MatriceScaAugLim}) satisfy $S_{12}^{\infty}\,S_{13}^{\infty}\,S_{23}^{\infty}\ne0$. Assume also that we have $s_{12}=0$ (zero transmission in $\Om_{\infty}$). Then the circle $\Gamma_{22}$ intersects the unit circle $\mathscr{C}$ at the point of affix $-1+0i$ (see Figure \ref{Scattering08pi}). 
\end{proposition}
\begin{proof}
If $s_{12}=0$, starting from (\ref{DefScatteringCoeff}), we see that there is some $\lambda\in\Cplx$ such that $u_{1}$ (defined in (\ref{DefScatteringCoeff})) admits the expansion 
\[
u_{1} = \chi_l\,(W^+_{1}+s_{11}\,W^-_{1}+\lambda\,(W^+_{2}-W^-_{2}))+\hat{u}_{1} 
\]
with $\hat{u}_{1}$ which decays as $O(e^{\sqrt{4\pi^2-k^2}x})$ for $x\to-\infty$ and as $O(e^{-\sqrt{\pi^2-k^2}x})$ for $x\to+\infty$. Define the symplectic (sesquilinear and anti-hermitian ($q(u,v)=-\overline{q(v,u)}$))  form $q(\cdot,\cdot)$ such that for all $u$, $v\in\mH^1_{\loc}(\Om_{\infty})$
\begin{equation}\label{DefSymplecticForm}
q(u,v)=\int_{\Sigma_{-2H}\cup\Sigma_{2H}} \cfrac{\partial u}{\partial n}\,\overline{v}-u\,\cfrac{\partial \overline{v}}{\partial n}\,d\sigma.
\end{equation}
Here $\Sigma_{\pm 2H}:=\{\pm 2H\}\times(0;1)$, $\partial_n=\pm\partial_x$ at $x=\pm 2H$ and $\mH^1_{\loc}(\Om_{\infty})$ refers to the Sobolev space of functions $\varphi$ such that $\varphi|_{\mathcal{O}}\in \mH^1(\mathcal{O})$ for all bounded domains $\mathcal{O}\subset\Om_{\infty}$. Using that $U_{2}^{\infty}$, $U_{3}^{\infty}$ and $u_1$ satisfy the homogeneous Helmholtz equation in $\Om_{\infty}$, integrating by parts, one obtains $q(U_{2}^{\infty},u_1)=q(U_{3}^{\infty},u_1)=0$. On the other hand, decomposing these three functions in Fourier series on $\Sigma_{\pm 2H}$, one finds 
\begin{eqnarray}
\label{Eqn1}0\ =\ q(U_{2}^{\infty},u_1)  & = & S^{\infty}_{21}\,\overline{s_{11}}-S^{\infty}_{22}\,\overline{\lambda}-\overline{\lambda}\\[4pt]
\label{Eqn2}0\ =\ q(U_{3}^{\infty},u_1)  & = & S^{\infty}_{31}\,\overline{s_{11}}-S^{\infty}_{32}\,\overline{\lambda}.
\end{eqnarray}
Coupling (\ref{Eqn1}) and (\ref{Eqn2}), we get the additional relation 
\begin{equation}\label{RelationT0}
S^{\infty}_{31}=\cfrac{S^{\infty}_{32}\,S^{\infty}_{21}}{1+S^{\infty}_{22}}
\end{equation}
for the coefficients of the augmented scattering matrix $\mathcal{S}^{\infty}$ when $s_{12}=0$. Note that since by assumption $S^{\infty}_{23}\ne0$, we have $S^{\infty}_{22}\ne-1$. Actually the latter property is true as soon as there is no trapped mode in the geometry $\Om_{\infty}$. Now, we explain how to use Identity (\ref{RelationT0})  to show that the circle $\Gamma_{22}$ defined in (\ref{setSPlusMoins}) passes through the point of affix $-1+0i$. From (\ref{eqnCenters}), we know that the center of $\Gamma_{22}$ is given by 
\[
Z_{22}=S^{\infty}_{22}+\dsp\frac{S^{\infty}_{23}\,\overline{S^{\infty}_{33}}\,S^{\infty}_{32}}{1-|S^{\infty}_{33}|^{2}}=\dsp\frac{S^{\infty}_{22}(|S^{\infty}_{31}|^{2}+|S^{\infty}_{32}|^{2})+
S^{\infty}_{23}\,\overline{S^{\infty}_{33}}\,S^{\infty}_{32}}{|S^{\infty}_{31}|^{2}+|S^{\infty}_{32}|^{2}}.
\]
Using that $S^{\infty}_{21}\overline{S^{\infty}_{31}}+
S^{\infty}_{22}\overline{S^{\infty}_{32}}+
S^{\infty}_{23}\overline{S^{\infty}_{33}}=0$ ($\mathcal{S}^{\infty}$ is unitary), we obtain
\[
Z_{22}=\dsp\frac{S^{\infty}_{22}\,|S^{\infty}_{31}|^{2}-S^{\infty}_{32}S^{\infty}_{21}\overline{S^{\infty}_{31}}}{|S^{\infty}_{31}|^{2}+|S^{\infty}_{32}|^{2}}.
\]
With (\ref{RelationT0}), we deduce 
\[
\begin{array}{lcl}
Z_{22}=\dsp\cfrac{|S^{\infty}_{32}|^{2}\,|S^{\infty}_{21}|^2}{|S^{\infty}_{31}|^{2}+|S^{\infty}_{32}|^{2}}\,\left(\cfrac{S^{\infty}_{22}}{|1+S^{\infty}_{22}|^2}-\cfrac{1}{1+\overline{S^{\infty}_{22}}}\right)&=&-\dsp\cfrac{|S^{\infty}_{32}|^{2}\,|S^{\infty}_{21}|^2}{|S^{\infty}_{31}|^{2}+|S^{\infty}_{32}|^{2}}\,\cfrac{1}{|1+S^{\infty}_{22}|^2}\\[14pt]
&=&-\dsp\cfrac{|S^{\infty}_{31}|^{2}}{|S^{\infty}_{31}|^{2}+|S^{\infty}_{32}|^{2}}\ =\ -1+P_{22}.
\end{array}
\]
The last equality in the above equation has been obtained using the expression of $P_{22}$ in (\ref{eqnRadius}). Since $P_{22}$ stands for the radius of $\Gamma_{22}$, this shows that $\Gamma_{22}$ passes through the point of affix $-1+0i$.
\end{proof}

\noindent Proposition \ref{PropositionThroughMinusOne} together with the error estimate $\|\mathcal{S}-\mathcal{S}^{\mrm{asy}}(\mathcal{L})\| \le C\,e^{-\beta \mathcal{L}}$ show that the curve $\mathcal{L}\mapsto S_{22}(\mathcal{L})$ passes as close as we wish to the point of affix $-1+0i$ as $\mathcal{L}\to+\infty$ when we know that $s_{12}=0$. Unfortunately, contrary to what has been done in the previous section to prove that $L\mapsto s_{12}(L)$ passes through zero as $L\to+\infty$ (see the proof of Theorem \ref{MainThmPart1}), the unitary structure of $\mathcal{S}$ is not sufficient to guarantee that $\mathcal{L}\mapsto S_{22}(\mathcal{L})$ passes exactly through the point of affix $-1+0i$  as $\mathcal{L}\to+\infty$. In our analysis, we will have to play also with the wavenumber $k$.

\begin{remark}
Let us present a simple calculation allowing one to feel that the situation $s_{12}=0$ in $\Om_{\infty}$ is interesting to construct trapped modes in $\Om_{\mathcal{L}}$. When $s_{12}=s_{21}=0$, the function $u_2$ in (\ref{DefScatteringCoeff}) admits the expansion 
\[
u_{2} = \chi_r\,(w^{-}_{2}+s_{22}\,w^{+}_{2})+\tilde{u}_2,
\]
where $\tilde{u}_2$ decays exponentially as $x\to\pm\infty$. Due to conservation of energy, $s_{12}=0$ implies $s_{22}=e^{i\eta}$ for some $\eta\in[0;2\pi)$. As a consequence, on $\{H+\mathcal{L}\}\times(0;1)$, we find $\partial_n(w^{-}_{2}+s_{22}\,w^{+}_{2})= ik(-e^{-ik(H+\mathcal{L})}+e^{i\eta} e^{ik(H+\mathcal{L})})/\sqrt{2k}$. Thus, there is a periodic sequence $(\mathcal{L}_n)$ such that $\partial_n(w^{-}_{2}+s_{22}\,w^{+}_{2})=0$ on $\{\mathcal{L}_n\}\times(0;1)$. Of course, this does not show that $u_{2}$ is a trapped mode in $\Omega_{\mathcal{L}_n}$ ($\partial_nu_2$ is exponentially small on $\{H+\mathcal{L}_n\}\times(0;1)$ but not zero). However it leads to think there is a trapped mode ``close to'' $(k,\Om_{\mathcal{L}_n})$. 
\end{remark}

\subsection{Proof of existence of trapped modes}

The coefficient $S_{22}$ depends both on $\mathcal{L}$ and $k$. Up to now, we have found pairs $(\mathcal{L},k)$ such that the dominant asymptotic term of $S_{22}$ is equal to $-1$. Now, we will show that there is a sequence $(\mathcal{L}_n,k_n)$ such that we have exactly $S_{22}=-1$ (indeed, from Lemma \ref{LemmaExistenceTrappedMode}, we know that this proves the existence of trapped modes).\\
\newline
Assume that there is $k_0\in(0;\pi)$ such that the transmission coefficient $s_{12}$ in (\ref{DefScatteringCoeff}) is zero. Assume also that the coefficients $S_{12}^{\infty}$, $S_{13}^{\infty}$, $S_{23}^{\infty}$ in (\ref{MatriceScaAugLim}) satisfy $S_{12}^{\infty}\,S_{13}^{\infty}\,S_{23}^{\infty}\ne0$ at the wavenumber $k_0$. Since $\mathcal{S}^{\infty}$ depend smoothly on the wavenumber, there is $\eps>0$ such that we have $S_{12}^{\infty}\,S_{13}^{\infty}\,S_{23}^{\infty}\ne0$ for all $k\in[k_0-\eps;k_0+\eps]$. For $k\in[k_0-\eps;k_0+\eps]$, as $\mathcal{L}\to+\infty$, we know from Proposition \ref{PropositionAugmScaUnit} that there is a sequence $(\mathcal{L}_n(k))$ such that $|S_{22}(\mathcal{L}_n(k),k)|=1$. Introduce $\alpha_n(k) \in[0;2\pi)$ such that 
\[
S_{22}(\mathcal{L}_n(k),k)=e^{i\alpha_n(k)}.
\]
The sequence $(\alpha_n(k))$ tends to $\alpha_{\infty}(k)$ where $\alpha_{\infty}(k)\in[0;2\pi)$ is such that 
\[
\Gamma_{22}(k)\cap\mathscr{C} =\{e^{i\alpha_{\infty}(k)}\}.
\]
Assume that 
\begin{equation}\label{HypoAbstraite}
\mbox{The map }k\mapsto \Im m\,Z_{22}(k)\mbox{ changes sign at $k=k_0$.}
\end{equation}
We remind the reader that $Z_{22}$ is the center of the circle $\Gamma_{22}$ defined in (\ref{eqnCenters}) such that $\Im m\,Z_{22}(k_0)=0$ (Proposition \ref{PropositionThroughMinusOne}). Observe that (\ref{HypoAbstraite}) is true for example if $\Im m\,\partial_k Z_{22} |_{k=k_0} \ne 0$. In this case, since $\alpha_{\infty}(k_0)=\pi$ (again Proposition \ref{PropositionThroughMinusOne}), we know that there is $\eps>0$ (smaller than the already introduced $\eps>0$) such that $\alpha_{\infty}(k_0-\eps)-\pi$ and $\alpha_{\infty}(k_0+\eps)-\pi$ have different signs. Pick $N\in\N$ large enough so that for all $n\ge N$, the quantities $\alpha_{n}(k_0-\eps)-\pi$ and $\alpha_{n}(k_0+\eps)-\pi$ have different signs. Since the map $k\mapsto \alpha_{n}(k)$ is continuous, we know that there is $k_n^{\star}\in[k_0-\eps;k_0+\eps]$ such that $\alpha_{n}(k_n^{\star})=\pi$. Then we have 
\[
S_{22}(\mathcal{L}_n(k_n^{\star}),k_n^{\star})=e^{i\alpha_n(k_n^{\star})}=-1.
\]
This shows the existence of trapped modes for Problem (\ref{PbInitial}) at the wavenumber $k_n^{\star}$ in the geometry $\Om_{\mathcal{L}_n(k^{\star}_n)}$. Note that as $n\to+\infty$, we have $k^{\star}_n\to k_0$. Moreover, $(\mathcal{L}_n(k^{\star}_n))$ is almost periodic of period $\pi/k$. We summarize these results in the following theorem.

\begin{theorem}\label{MainThmPart2}
Assume that there holds $s_{12}=0$ (zero transmission in $\Om_{\infty}$) at the wavenumber $k_0\in(0;\pi)$. Assume also that the coefficients $S_{12}^{\infty}$, $S_{13}^{\infty}$, $S_{23}^{\infty}$ in (\ref{MatriceScaAugLim}) satisfy $S_{12}^{\infty}\,S_{13}^{\infty}\,S_{23}^{\infty}\ne0$ at the wavenumber $k_0$ and that (\ref{HypoAbstraite}) is true. Then there are sequences $(\mathcal{L}_n)$, $(k_n)$ such that Problem (\ref{PbInitial}) admits trapped modes in the geometry $\Om_{\mathcal{L}_n}$ at the wavenumber $k_n$. Moreover, there hold $\lim_{n\to+\infty} \mathcal{L}_n=+\infty$ and $\lim_{n\to+\infty} k_n=k_0$.
\end{theorem}
\begin{remark}
In Section \ref{SectionNumExpe}, \S\ref{paragraphTrappedModes} below, we provide an example of situation where numerically the assumptions of Theorem \ref{MainThmPart2} are satisfied. 
\end{remark}

\section{Numerical experiments}\label{SectionNumExpe}

\subsection{Zero transmission}\label{paragraphPerfectRef}

In the first series of experiments, we set $M:=(0.2,0.4)$ and 
\[
\om_L :=\{(x,y)\in\R\times (0;1)\cup (-1/2;1/2)\times [1;L)\}\setminus \overline{B(M,0.3)}
\]
(see Figure \ref{GeomTZero}). Here $B(M,0.3)$ corresponds to the open ball centered at $M$ of radius $0.3$. Note that there is no symmetry in the geometry. For each $L$ in a given range, we compute numerically the  coefficients of the scattering matrix $\mathbb{S}\in\Cplx^{2\times2}$ defined in (\ref{DefMatrixScattering}). To proceed, we use a $\mrm{P}2$ finite element method in a truncated waveguide. On the artificial boundary created by the truncation, a Dirichlet-to-Neumann operator with $\mrm{15}$ terms serves as a transparent condition. We take $k=0.8\pi$ and $\ell=1$. In Figure \ref{MatriceScatteringTZero}, we display the scattering coefficients for $L\in(1.1;6)$. In accordance with the results obtained in \S\ref{paragraphAsymptoAnalysis}, we observe that when $L\to+\infty$, asymptotically $L\mapsto s_{11}(L)$, $L\mapsto s_{12}(L)$ and $L\mapsto s_{22}(L)$ run on circles. More precisely, computing the coefficients of the scattering matrix $\mathbb{S}^{\infty}\in\mathbb{C}^{3\times 3}$ defined in (\ref{UnboundedScatteringMatrix}), we indeed check that asymptotically $L\mapsto s_{11}(L)$, $L\mapsto s_{12}(L)$ and $L\mapsto s_{22}(L)$ run respectively on the circles $\gamma_{11}$, $\gamma_{12}$ and $\gamma_{22}$ obtained in (\ref{setSPlusMoins}). The asymptotic sets $\gamma_{11}$, $\gamma_{12}$, $\gamma_{22}$ are displayed in Figure \ref{MatriceScatteringTZero} but are mostly covered by the marks of $s_{11}(L)$, $s_{12}(L)$, $s_{22}(L)$ (we remind the reader that the convergence is exponentially fast). We also note that, as predicted by Theorem \ref{MainThmPart1}, the curve $L\mapsto s_{12}(L)$ indeed passes through zero as $L\to+\infty$.\\
In Figure \ref{logT}, we display the curve $L\mapsto -\ln |s_{12}(L)|$ for $L\in(1.1;6)$. The peaks correspond to the values of $L$ such that $s_{12}(L)= 0$ (zero transmission). According to the proof of Theorem \ref{MainThmPart1}, we expect that the peaks are almost periodic with a distance between two peaks tending to $\pi/k=1.25$ as $L\to+\infty$. The numerical results we get are coherent with this value. Finally, in Figure \ref{GeomTZero}, we represent the real part of the total field $u_1$ defined in (\ref{DefScatteringCoeff}) for $L=2.496$ (second peak of Figure \ref{logT}). We can observe that the field is indeed exponentially decaying as $x\to+\infty$.

\begin{figure}[!ht]
\centering
\includegraphics[scale=0.75]{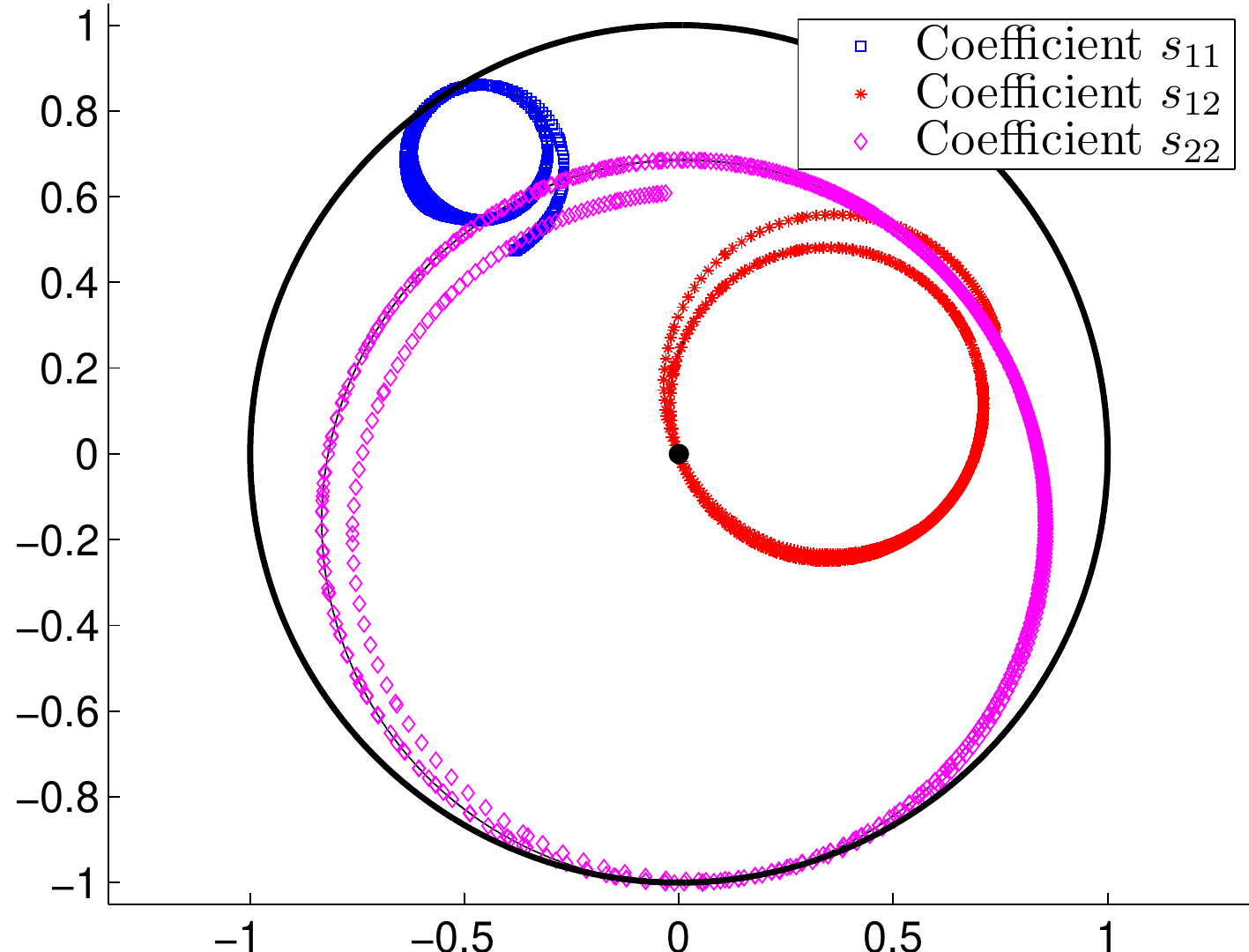}
\caption{Coefficients $L\mapsto s_{11}(L)$, $L\mapsto s_{12}(L)$ and $L\mapsto s_{22}(L)$ for $L\in(1.1;6)$. The thin black circles (mostly hidden by the symbols) correspond to the asymptotic circles $\gamma_{11}$, $\gamma_{12}$ and  $\gamma_{22}$ defined in (\ref{setSPlusMoins}). According to the conservation of energy, we know that the scattering coefficients are located inside the unit disk marked by the black bold line. \label{MatriceScatteringTZero}}
\end{figure}

\begin{figure}[!ht]
\centering
\includegraphics[trim=0 2.7cm 0 3cm, clip,scale=1]{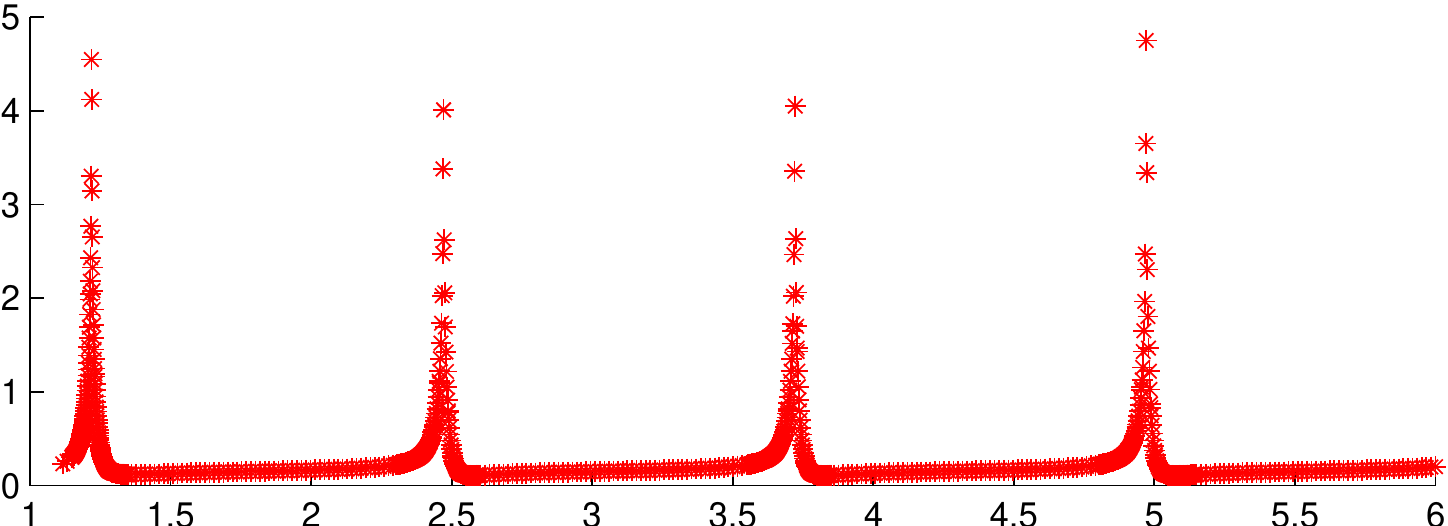}
\caption{Curve $L\mapsto -\ln |s_{12}(L)|$ for $L\in(1.1;6)$. The peaks correspond to the values of $L$ for which there is zero transmission in $\om_L$.\label{logT}}
\end{figure}

\begin{figure}[!ht]
\centering
\includegraphics[width=0.96\textwidth]{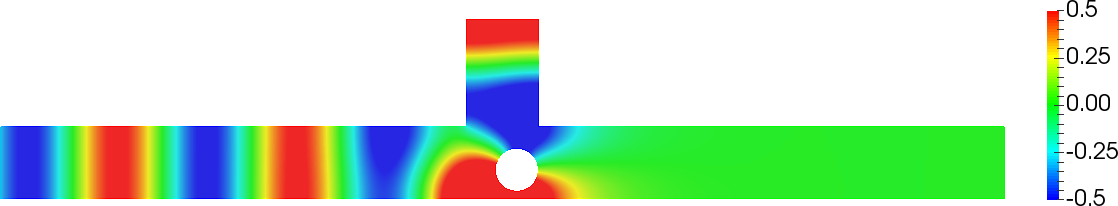}
\caption{Real part of the total field $u_1$ defined in (\ref{DefScatteringCoeff}) for a setting where $s_{12}(L)=0$ ($L=2.496$). The incident field is coming from the left. \label{GeomTZero}}
\end{figure}

\begin{figure}[!ht]
\centering
\includegraphics[width=0.83\textwidth]{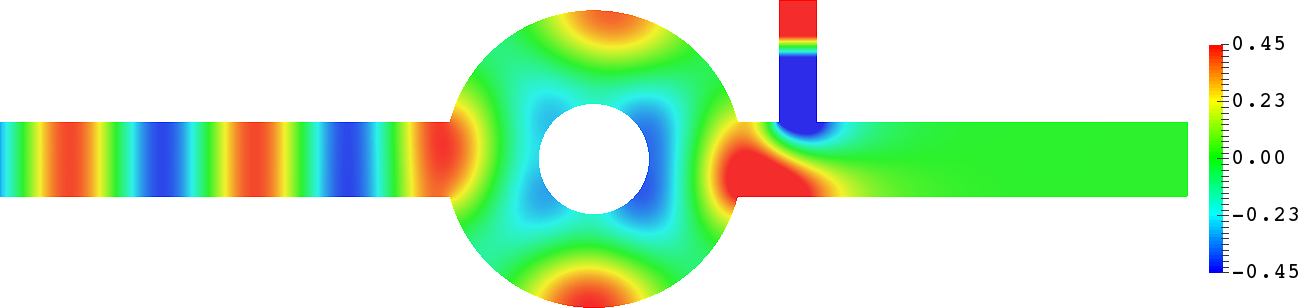}\hspace{-2.6cm}
\includegraphics[width=0.3\textwidth]{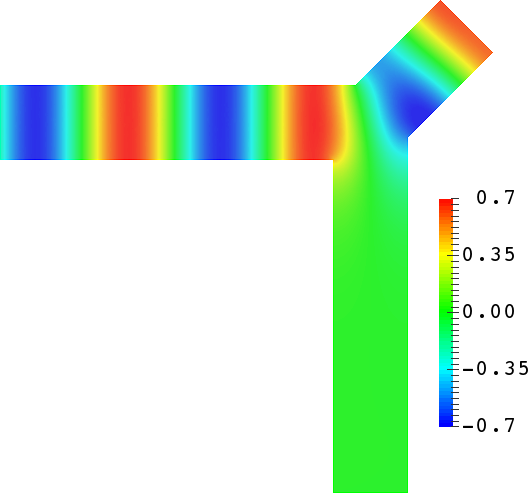}
\caption{Real part of the total field in two other geometries where the transmission coefficient is zero. For the picture on left, we played with the height of the vertical branch. For the picture on right, we played with the length of the diagonal branch (the waveguide is unbounded in the left and down directions). The incident field is coming from the left.\label{OtherGeom}}
\end{figure}

\subsection{Trapped modes}\label{paragraphTrappedModes}
In the second series of experiments, we set $M:=(0.2,0.4)$ as in \S\ref{paragraphPerfectRef}, $L:=2.496$ (this is the value obtained in the numerical experiments leading to Figure \ref{GeomTZero}) and, for $\mathcal{L}>0$,
\[
\Om_\mathcal{L} :=\{(x,y)\in(-\infty;1/2+\mathcal{L})\times (0;1)\cup (-1/2;1/2)\times [1;L)\}\setminus \overline{B(M,0.3)}.
\]
The domain $\Om_\mathcal{L}$ is pictured in Figure \ref{TrappedMode1}. According to the results of \S\ref{paragraphPerfectRef}, we know that $s_{12}=0$ in $\Om_{\infty}$ (zero transmission) at the wavenumber $k=0.8\pi$. For each $\mathcal{L}$ in a given range, we compute numerically the  coefficients of the augmented scattering matrix $\mathcal{S}\in\Cplx^{2\times2}$ defined in (\ref{DefAugmentedMatrixScattering}). To proceed, again we use a $\mrm{P}2$ finite element method set in a truncated waveguide. We emphasize here that we need to work with a well-suited Dirichlet-to-Neumann map to deal with the wave packet $W_2^+$ appearing in the decompositions of $U_1$, $U_2$ in (\ref{DefSolutionScaAug}). In Figures \ref{Scattering078pi}, \ref{Scattering08pi}, \ref{Scattering082pi}, we display the coefficients $\mathcal{L}\mapsto S_{11}(\mathcal{L})$, $\mathcal{L}\mapsto S_{12}(\mathcal{L})$, $\mathcal{L}\mapsto S_{22}(\mathcal{L})$ for $\mathcal{L}\in(0.1;3.5)$ and respectively $k=0.78\pi$, $k=0.8\pi$, $k=0.82\pi$. In these figures, we also display the asymptotic circles $\Gamma_{11}$, $\Gamma_{12}$ and $\Gamma_{22}$ defined in (\ref{setSPlusMoins}). In accordance with the results of \S\ref{paragraphAsymptoCoefAug}, we observe that asymptotically as $L\to+\infty$, the coefficients $\mathcal{L}\mapsto S_{11}(\mathcal{L})$, $\mathcal{L}\mapsto S_{12}(\mathcal{L})$, $\mathcal{L}\mapsto S_{22}(\mathcal{L})$ run respectively on $\Gamma_{11}$, $\Gamma_{12}$, $\Gamma_{22}$. Moreover, the curves $\mathcal{L}\mapsto S_{12}(\mathcal{L})$ pass through zero (Proposition \ref{PropositionAugmScaUnit}). In Figure \ref{Scattering08pi}, we see that for $k=0.8\pi$, the circle $\Gamma_{22}$ passes through the point of affix $-1+0i$ as shown in Proposition \ref{PropositionThroughMinusOne}. 
Comparing Figures \ref{Scattering078pi}, \ref{Scattering08pi} and \ref{Scattering082pi}, we observe that the center $Z_{22}$ of $\Gamma_{22}$ passes from the upper half plane to the lower half plane as $k$ goes from $0.8\pi^{-}$ to $0.8\pi^{+}$. As a consequence, we are tempted to think that the map $k\mapsto \Im m\,Z_{22}(k)$ changes sign at $k=0.8\pi$ (Assumption (\ref{HypoAbstraite})). In Figure \ref{TrappedMode1} we display a trapped mode in $\Om_{\mathcal{L}}$ for $\mathcal{L}=1.354$ at the wavenumber $k=2.512... \approx 0.8\pi$. Figure \ref{TrappedMode2} represents the symmetrised version with respect to the line $\{x=H+\mathcal{L}\}$ of the trapped mode of Figure \ref{TrappedMode1}.

\begin{figure}[!ht]
\centering
\includegraphics[scale=0.8]{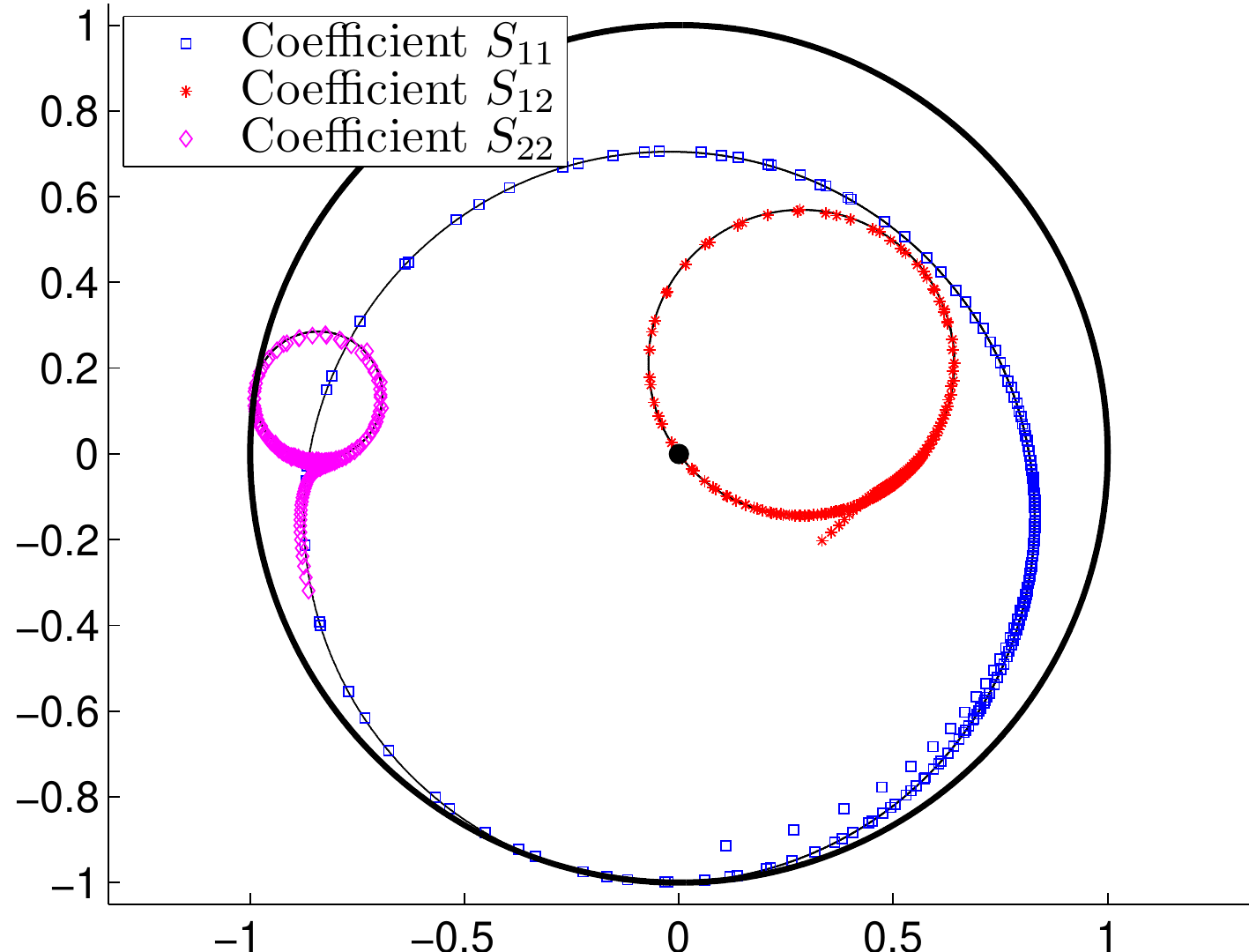}
\caption{Coefficients $\mathcal{L}\mapsto S_{11}(\mathcal{L})$, $\mathcal{L}\mapsto S_{12}(\mathcal{L})$ and $\mathcal{L}\mapsto S_{22}(\mathcal{L})$ for $\mathcal{L}\in(0.1;3.5)$. The thin black circles (mostly hidden by the symbols) correspond to the asymptotic circles $\Gamma_{11}$, $\Gamma_{12}$ and  $\Gamma_{22}$ defined in (\ref{setSPlusMoins}). Here $k=0.78\pi$. \label{Scattering078pi}}
\end{figure}

\begin{figure}[!ht]
\centering
\includegraphics[scale=0.8]{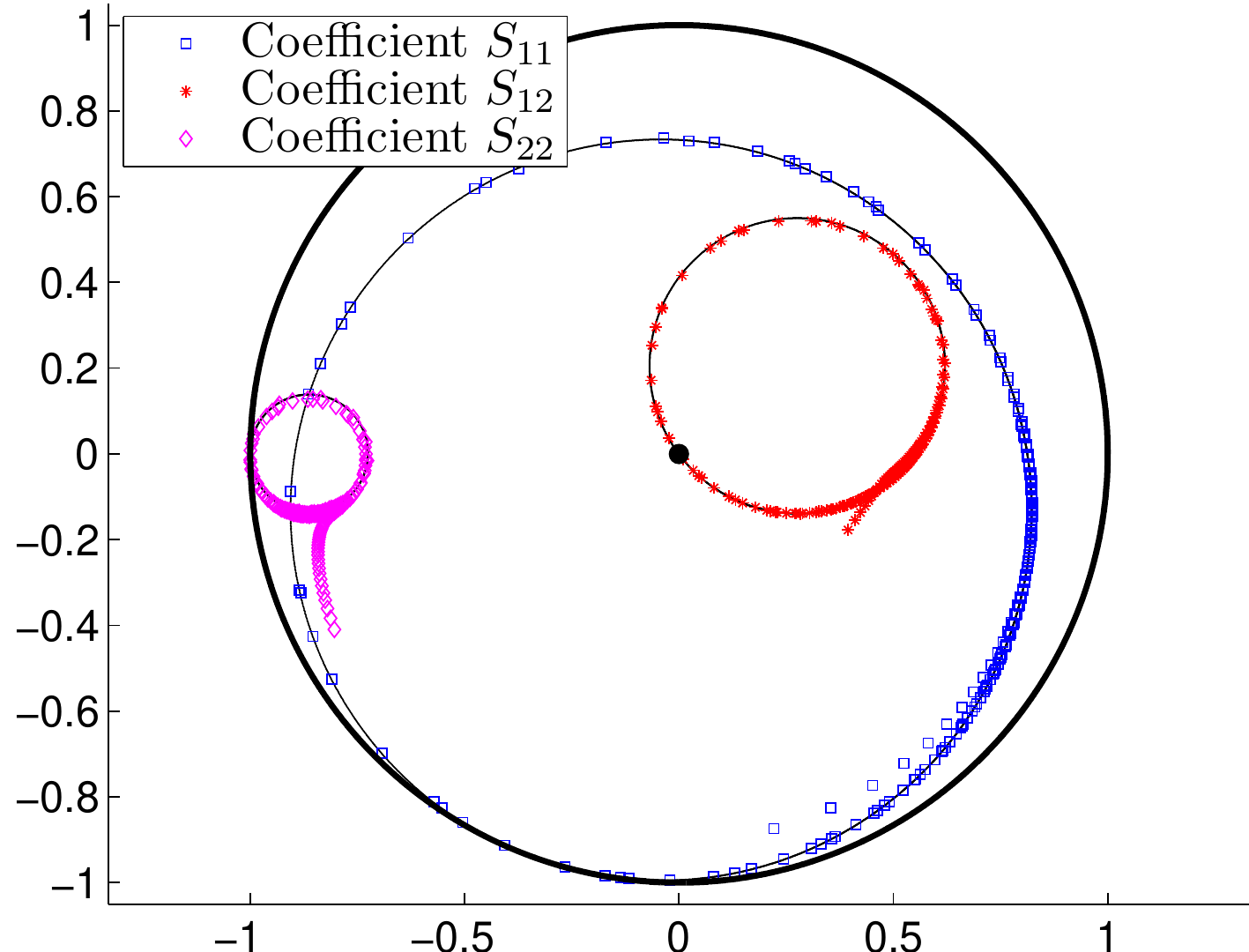}
\caption{Coefficients $\mathcal{L}\mapsto S_{11}(\mathcal{L})$, $\mathcal{L}\mapsto S_{12}(\mathcal{L})$ and $\mathcal{L}\mapsto S_{22}(\mathcal{L})$ for $\mathcal{L}\in(0.1;3.5)$. The thin black circles (mostly hidden by the symbols) correspond to the asymptotic circles $\Gamma_{11}$, $\Gamma_{12}$ and  $\Gamma_{22}$ defined in (\ref{setSPlusMoins}). Here $k=0.8\pi$. \label{Scattering08pi}}
\end{figure}

\begin{figure}[!ht]
\centering
\includegraphics[scale=0.8]{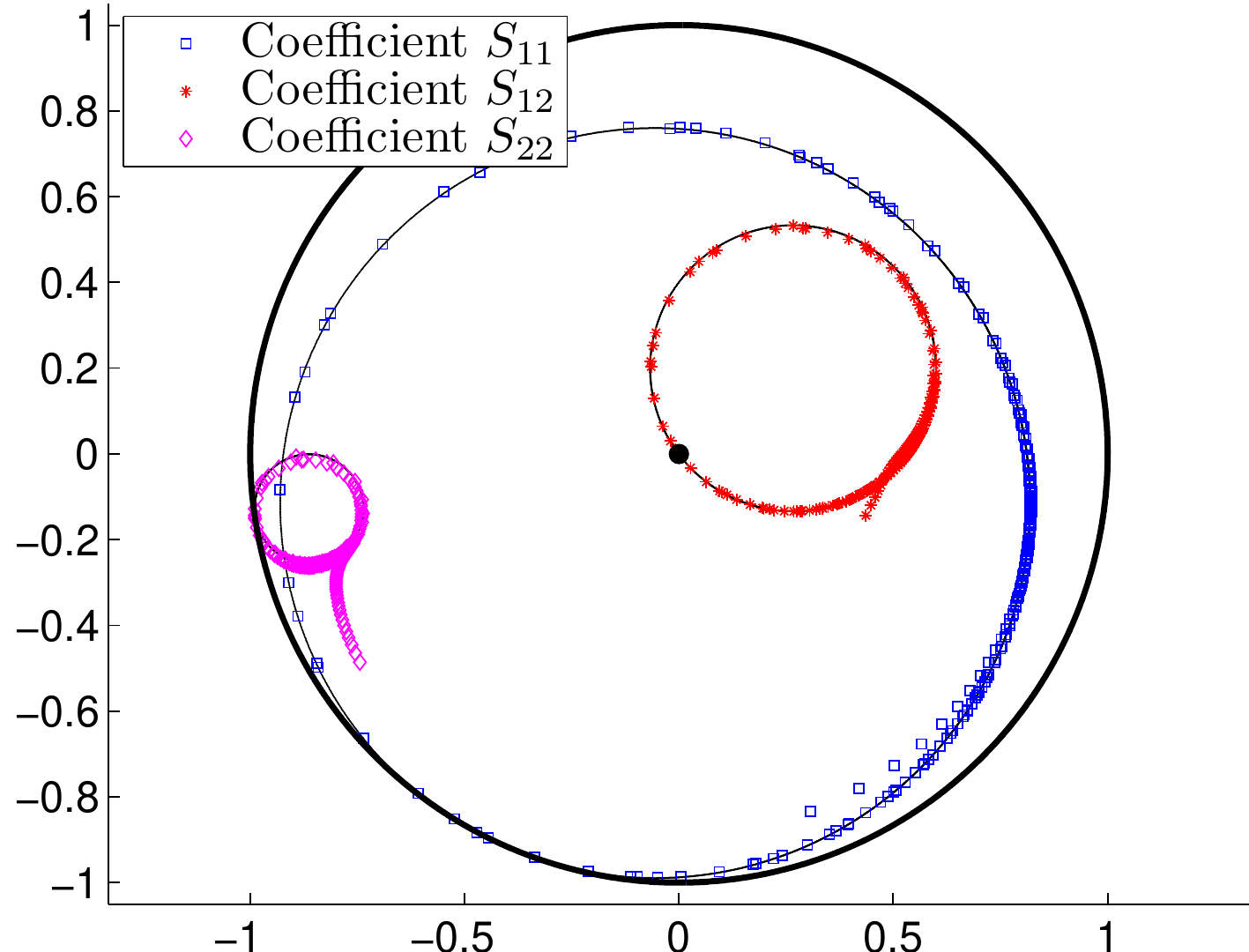}
\caption{Coefficients $\mathcal{L}\mapsto S_{11}(\mathcal{L})$, $\mathcal{L}\mapsto S_{12}(\mathcal{L})$ and $\mathcal{L}\mapsto S_{22}(\mathcal{L})$ for $\mathcal{L}\in(0.1;3.5)$. The thin black circles (mostly hidden by the symbols) correspond to the asymptotic circles $\Gamma_{11}$, $\Gamma_{12}$ and  $\Gamma_{22}$ defined in (\ref{setSPlusMoins}). Here $k=0.82\pi$. \label{Scattering082pi}}
\end{figure}

\begin{figure}[!ht]
\centering
\includegraphics[scale=0.5]{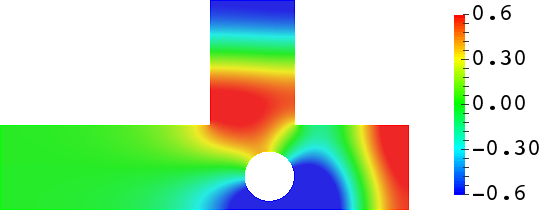}
\caption{Trapped mode for Problem (\ref{PbInitial}) in $\Om_{\mathcal{L}}$. Here $k=2.5125645\approx 0.8\pi$ and $\mathcal{L}=1.354$. \label{TrappedMode1}}
\end{figure}

\begin{figure}[!ht]
\centering
\includegraphics[scale=0.5]{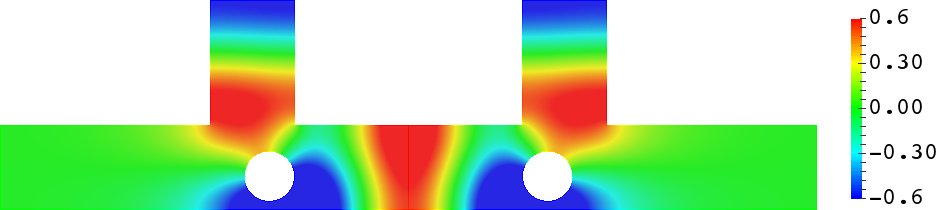}
\caption{Symmetrization of the trapped mode of Figure \ref{TrappedMode1}. \label{TrappedMode2}}
\end{figure}

\newpage
\clearpage

\section{Conclusion}\label{SectionConclusion}

In this article, first we explained how to construct waveguides with two open channels such that the transmission coefficient in monomode regime is zero. In that case, the energy carried by an ingoing mode propagating in one channel is completely backscattered, this is the mirror effect. The principal novelty of this study is that there is no assumption of symmetry of the geometry. Then in a second step, from a geometry where it is known that the transmission coefficient is zero, truncating one of the channel, we showed how to create half-waveguides supporting trapped modes. All the techniques presented here can be adapted in higher dimension (waveguides of $\R^d$ with $d\ge3$). Moreover, Neumann boundary conditions can be replaced by Dirichlet boundary conditions, the developments are exactly the same. For the construction of completely reflecting geometries, we have assumed that outside a compact domain, the waveguide coincides with the strip $\R\times(0;1)$. This is not needed in the analysis and configurations like the ones of Figure \ref{OtherGeom} right or Figure \ref{FigConclu} top can be considered as well. More precisely, the two open channels can be oriented in different directions and their height can differ. Similarly, to provide examples of half-waveguides supporting trapped modes, we can start from geometries as the ones of Figure \ref{FigConclu} top such that the transmission coefficient is zero and truncate one branch.\\ 
To continue this work, there are many directions to investigate. Dealing with the case of waveguides with $N$ open channels for $N\ge3$ (see Figure \ref{FigConclu} bottom) seems an interesting one. For this problem, again several questions can be considered. For example, can one find a geometry such that the energy carried by an ingoing mode propagating in one channel is completely backscattered? In this case, we have to cancel $N-1$ transmission coefficients ($s_{12},\dots,s_{1N}$). Playing with one branch is probably not enough. Can we do it with $N-1$ branches? And then, starting from a geometry like the one of Figure \ref{FigConclu} bottom where the mirror effect occurs ($s_{12}=\cdots=s_{1N}=0$), truncating channel 1, can we create waveguides supporting trapped modes? This sounds more easily achievable. Another interesting direction would be to study what happens at higher wavenumber when several modes can propagate. For the moment, this is highly open.

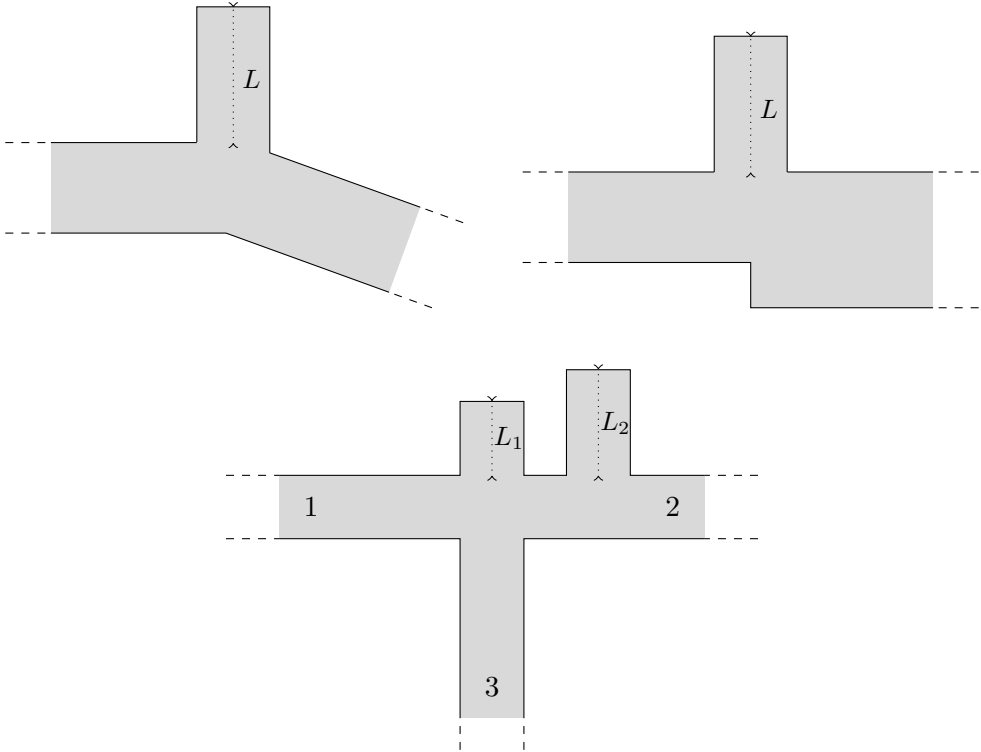
\begin{figure}[!ht]
\centering
\begin{tikzpicture}[scale=1.2]
\draw[fill=gray!30](-0.4,-0.4) rectangle (0.4,2);
\draw[fill=gray!30,draw=none,rotate=0](-2,-0.5) rectangle (0,0.5);
\draw[fill=gray!30,draw=none,rotate=-20](0,-0.5) rectangle (2,0.5);
\draw[dashed,rotate=-20] (2.5,0.5)--(2,0.5); 
\draw[dashed,rotate=-20] (2.5,-0.5)--(2,-0.5); 
\draw[dashed] (-2.5,0.5)--(-2,0.5); 
\draw[dashed] (-2.5,-0.5)--(-2,-0.5); 
\draw (-2,0.5)--(-0.4,0.5);
\draw[rotate=-20] (2,0.5)--(0.24,0.5);
\draw (-2,-0.5)--(-0.08,-0.5) [rotate=-20]-- (2,-0.5);
\draw[dotted,>-<] (0,0.44)--(0,2.05);
\node at (0.2,1.2){\small $L$};
\end{tikzpicture}\qquad\begin{tikzpicture}[scale=1.2]
\draw[fill=gray!30](-0.4,-0.4) rectangle (0.4,2);
\draw[fill=gray!30,draw=none,rotate=0](-2,-0.5) rectangle (0,0.5);
\draw[fill=gray!30,draw=none](0,-0.5) rectangle (2,0.5);
\draw[fill=gray!30,draw=none](0,-1) rectangle (2,0.5);
\draw[dashed] (-2.5,0.5)--(-2,0.5); 
\draw[dashed] (-2.5,-0.5)--(-2,-0.5); 
\draw[dashed] (2.5,0.5)--(2,0.5); 
\draw[dashed] (2.5,-1)--(2,-1); 
\draw (-2,0.5)--(-0.4,0.5);
\draw (2,0.5)--(0.4,0.5);
\draw (-2,-0.5)-- (0,-0.5)--(0,-1)--(2,-1);
\draw[dotted,>-<] (0,0.44)--(0,2.05);
\node at (0.2,1.2){\small $L$};
\end{tikzpicture}\\[20pt]
\raisebox{2mm}{\begin{tikzpicture}[scale=1.4]
\draw[fill=gray!30,draw=none](-2,-0.3) rectangle (2,0.3);
\draw[fill=gray!30,draw=none](-0.3,-2) rectangle (0.3,1);
\draw[fill=gray!30,draw=none](0.7,0) rectangle (1.3,1.3);
\draw (-2,0.3)--(-0.3,0.3)--(-0.3,1)--(0.3,1)--(0.3,0.3)--(0.7,0.3)--
(0.7,1.3)--(1.3,1.3)--(1.3,0.3)--(2,0.3);
\draw (-2,-0.3)--(-0.3,-0.3)--(-0.3,-2);
\draw (2,-0.3)--(0.3,-0.3)--(0.3,-2);
\draw[dotted,>-<] (0,0.25)--(0,1.05);
\node at (0.16,0.65){\small $L_1$};
\draw[dotted,>-<] (1,0.25)--(1,1.35);
\node at (1.16,0.8){\small $L_2$};
\node at (-1.7,0){ $1$};
\node at (1.7,0){ $2$};
\node at (0,-1.7){ $3$};

\draw[dashed] (2.5,0.3)--(2,0.3); 
\draw[dashed] (2.5,-0.3)--(2,-0.3); 
\draw[dashed] (-2.5,0.3)--(-2,0.3); 
\draw[dashed] (-2.5,-0.3)--(-2,-0.3);
\draw[dashed] (0.3,-2.3)--(0.3,-2); 
\draw[dashed] (-0.3,-2.3)--(-0.3,-2);
\end{tikzpicture}}
\caption{Top: other types of waveguides where the approaches presented above work. Bottom: example of waveguide with three open channels. \label{FigConclu}} 
\end{figure} 

\newpage

\noindent In Figure \ref{DomainCircle}, we represent another geometry $\om_L$ where the branch of finite length used in the article has been replaced by a half disk of radius $L$. We also add a fixed non penetrable (Neumann) obstacle in the domain to break the symmetry of the geometry. In Figure \ref{MatriceScatteringTZeroCircle} left, we computed the scattering coefficients (see (\ref{DefScatteringCoeff})) with respect to $L\in(0.5;2)$. Again we observe that even though $\om_L$ is completely not symmetric, the transmission coefficient passes through zero as $L$ increases. This is confirmed when we draw the curve $L\mapsto -\ln |s_{12}(L)|$ (Figure \ref{MatriceScatteringTZeroCircle} top right). In Figure \ref{MatriceScatteringTZeroCircle} bottom right, we represent the real part of the total field in a situation where the transmission coefficient is zero. All that results are numerical results. It would be interesting to prove them rigorously. However the asymptotic analysis for this problem is different from the one presented above.

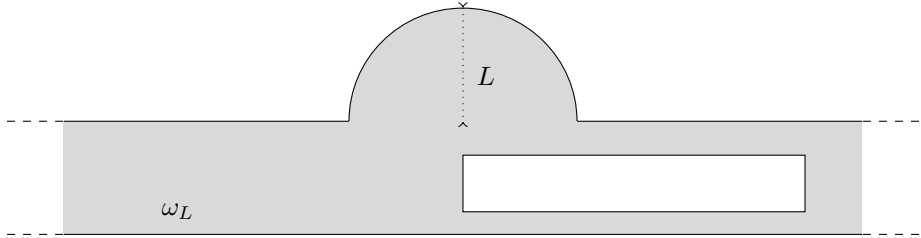
\begin{figure}[!ht]
\centering
\begin{tikzpicture}[scale=1.5]
\draw[fill=gray!30] (1,1) arc (0:180:1);
\draw[fill=gray!30,draw=none](-3.5,0) rectangle (3.5,1);
\draw[fill=white](0,0.2) rectangle (3,0.7);
\draw (-3.5,0)--(3.5,0); 
\draw (-3.5,1)--(-1,1); 
\draw (3.5,1)--(1,1); 
\draw[dashed] (4,0)--(3.5,0); 
\draw[dashed] (4,1)--(3.5,1);
\draw[dashed] (-4,0)--(-3.5,0); 
\draw[dashed] (-4,1)--(-3.5,1);
\draw[dotted,>-<] (0,0.95)--(0,2.05);
\node at (0.2,1.4){\small $L$};
\node at (-2.5,0.2){\small $\om_{L}$};
\end{tikzpicture}
\caption{Geometry of $\om_{L}$. \label{DomainCircle}} 
\end{figure}

\begin{figure}[!ht]
\centering
\includegraphics[trim=0 0cm 1.5cm 0cm, clip,scale=0.54]{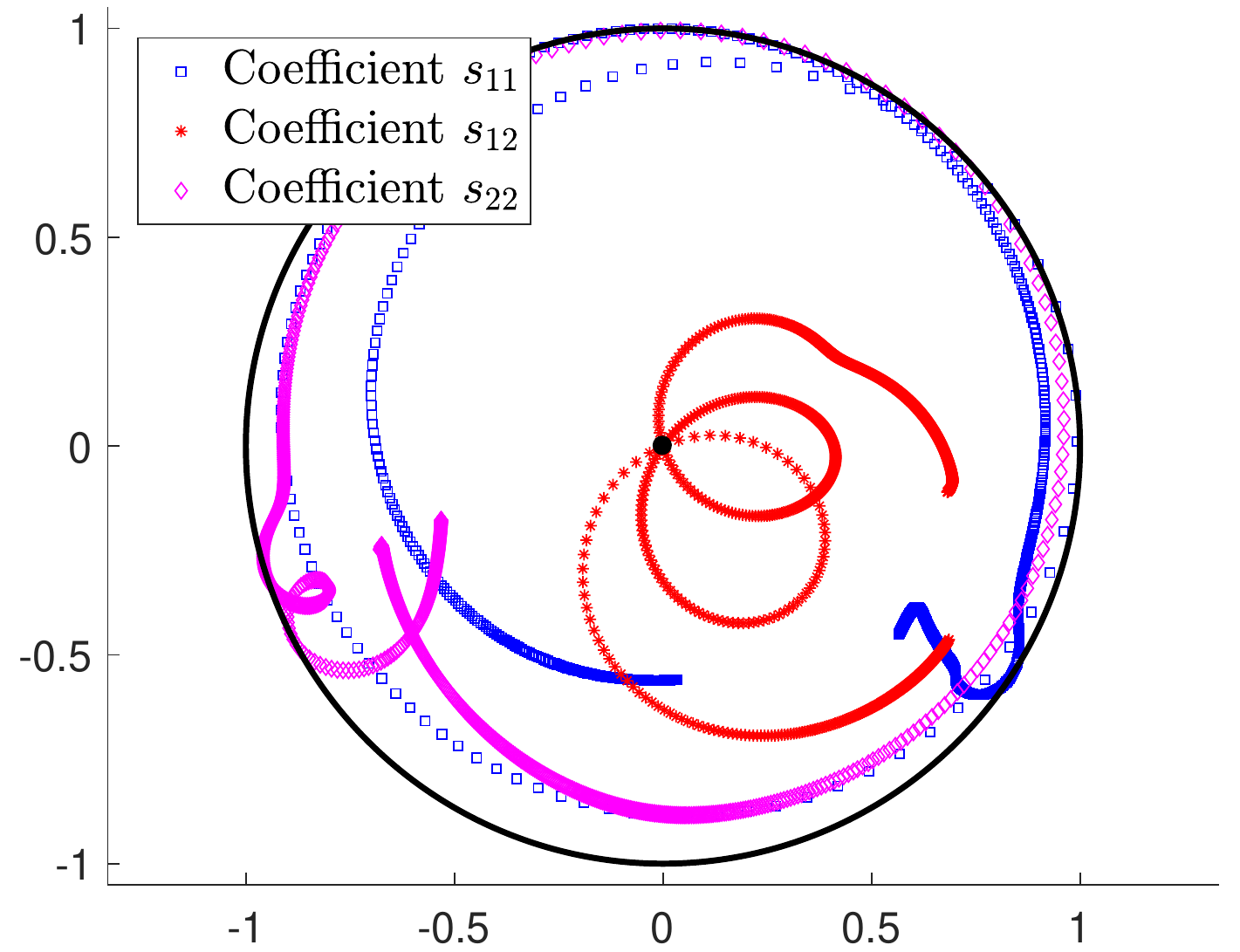}\quad\raisebox{2.7cm}{\begin{tabular}{cc}
\includegraphics[trim=0 2.6cm 0 3cm, clip,width=8.5cm]{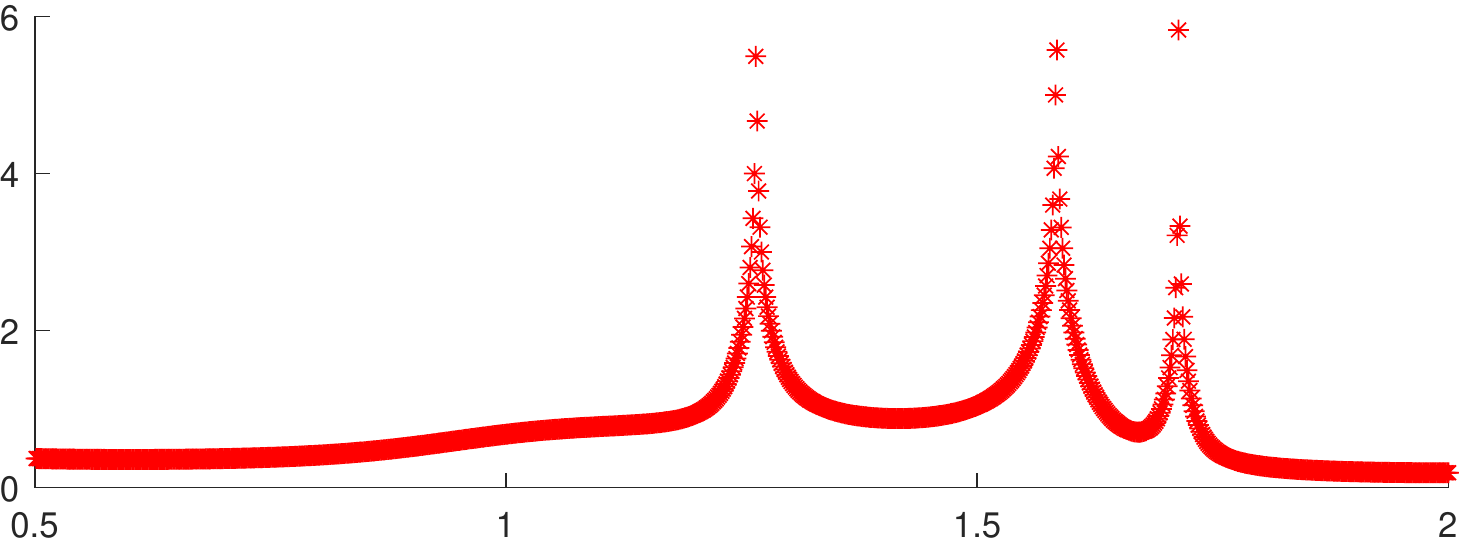}\\[8pt]
\includegraphics[width=8.5cm]{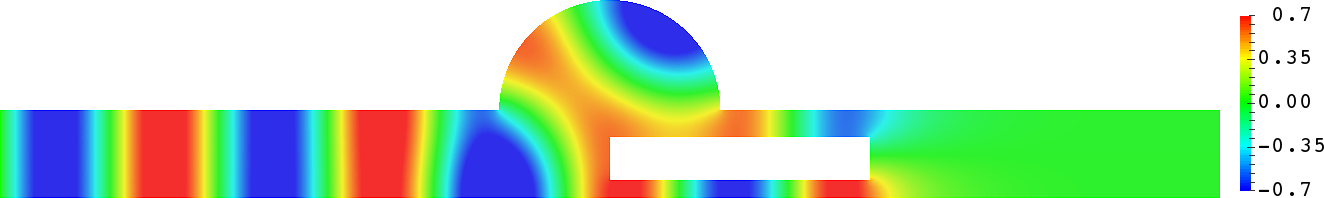}
\end{tabular}}
\caption{Left: coefficients $L\mapsto s_{11}(L)$, $L\mapsto s_{12}(L)$ and $L\mapsto s_{22}(L)$ for $L\in(0.5;2)$ in the geometry of Figure \ref{DomainCircle}. According to the conservation of energy, we know that the scattering coefficients are located inside the unit disk marked by the black bold line. Top right: curve $L\mapsto -\ln |s_{12}(L)|$ for $L\in(0.5;2)$. Bottom right: real part of the total field $u_1$ defined in (\ref{DefScatteringCoeff}) for a setting where $s_{12}(L)=0$ ($L=1.265488$). The incident field is coming from the left. \label{MatriceScatteringTZeroCircle}}
\end{figure}

\section{Annex}

\subsection{Properties of the scattering matrices}
In this paragraph, for the convenience of the reader, we provide the details of the proofs of two results used in the previous analysis. We start with a well-known lemma.
\begin{lemma}\label{LemmaUnitary}
The scattering matrix $\mathbb{S}$ defined in (\ref{DefMatrixScattering}) is unitary and symmetric. 
\end{lemma}
\begin{proof}
Let us work with the symplectic form $q(\cdot,\cdot)$ introduced in (\ref{DefSymplecticForm}) such that 
\[
q(u,v)=\int_{\Sigma_{-2H}\cup\Sigma_{2H}} \cfrac{\partial u}{\partial n}\,\overline{v}-u\,\cfrac{\partial \overline{v}}{\partial n}\,d\sigma,\qquad\forall u,v\in\mH^1_{\loc}(\Om_L).
\]
Integrating by parts and using that the functions $u_1$, $u_2$ defined in (\ref{DefScatteringCoeff}) satisfy the Helmholtz equation, we obtain $q(u_i,u_j)=0$ for $i,\,j\in\{1,2\}$. On the other hand, decomposing $u_1$, $u_2$ in Fourier series on $\Sigma_{\pm 2H}$, we find 
\[
\begin{array}{c}
q(u_1,u_1) = (-1+|s_{11}|^2+|s_{12}|^2)\,i,\quad q(u_2,u_2) = (-1+|s_{22}|^2+|s_{21}|^2)\,i\\[4pt]
q(u_1,u_2)=-\overline{q(u_2,u_1)}=s_{11}\overline{s_{21}}+s_{12}\overline{s_{22}}.
\end{array}
\]
These relations allow us to prove that $\mathbb{S}\,\overline{\mathbb{S}}^{\top}=\mrm{Id}^{2\times 2}$, that is to conclude that $\mathbb{S}$ is unitary. On the other hand, one finds $q(u_1,\overline{u_2})=0=-s_{21}+s_{12}$. We deduce that $\mathbb{S}$ is symmetric.
\end{proof}

\begin{lemma}\label{lemmaUnitary}
For all $L>1$, the matrix $\mathbb{S}^{\mrm{asy}}(L)\in\Cplx^{2\times2}$ defined in (\ref{defMatriceLim}) is unitary.
\end{lemma}
\begin{proof}
A direct computation using the definitions of the coefficients of $\mathbb{S}^{\mrm{asy}}(L)$ (see (\ref{defCoeffLim})) gives
\[
\begin{array}{lcl}
|s^{\mrm{asy}}_{11}|^2+|s^{\mrm{asy}}_{12}|^2&=&|s^{\infty}_{11}|^2+|s^{\infty}_{12}|^2+\cfrac{|s^{\infty}_{13}|^2\,(|s^{\infty}_{13}|^2+|s^{\infty}_{23}|^2)}{|e^{-2ikL}-s^{\infty}_{33}|^2}+2\Re e\,\cfrac{s^{\infty}_{13}\,(\overline{s^{\infty}_{11}}\,s^{\infty}_{13}+\overline{s^{\infty}_{12}}\,s^{\infty}_{23})}{e^{-2ikL}-s^{\infty}_{33}}\\[12pt]
&=&|s^{\infty}_{11}|^2+|s^{\infty}_{12}|^2+\cfrac{|s^{\infty}_{13}|^2\,(1-|s^{\infty}_{33}|^2)}{|e^{-2ikL}-s^{\infty}_{33}|^2}-2|s^{\infty}_{13}|^2\Re e\,\cfrac{\overline{s^{\infty}_{33}}}{e^{-2ikL}-s^{\infty}_{33}}\\[12pt]
&=&|s^{\infty}_{11}|^2+|s^{\infty}_{12}|^2+|s^{\infty}_{13}|^2\,\cfrac{|e^{-2ikL}-s^{\infty}_{33}|^2}{|e^{-2ikL}-s^{\infty}_{33}|^2}\ =\ 1.
\end{array}
\]
To obtain the equalities above, we used several times the fact that $\mathbb{S}^{\infty}\in\Cplx^{3\times3}$ is unitary. Analogously, one finds $|s^{\mrm{asy}}_{12}|^2+|s^{\mrm{asy}}_{22}|^2=1$. On the other hand, we have
\[
\begin{array}{ll}
& s^{\mrm{asy}}_{11}\,\overline{s^{\mrm{asy}}_{12}}+s^{\mrm{asy}}_{12}\,\overline{s^{\mrm{asy}}_{22}}\\[6pt]
=&s^{\infty}_{11}\,\overline{s^{\infty}_{12}}+s^{\infty}_{12}\,\overline{s^{\infty}_{22}}+\cfrac{|s^{\infty}_{13}|^2\,s^{\infty}_{13}\,\overline{s^{\infty}_{23}}+|s^{\infty}_{23}|^2\,s^{\infty}_{13}\,\overline{s^{\infty}_{23}}}{|e^{-2ikL}-s^{\infty}_{33}|^2} +\cfrac{s^{\infty}_{13}\,(s^{\infty}_{13}\,\overline{s^{\infty}_{12}}+s^{\infty}_{23}\,\overline{s^{\infty}_{22}})}{e^{-2ikL}-s^{\infty}_{33}}+\cfrac{\overline{s^{\infty}_{23}}\,(s^{\infty}_{11}\,\overline{s^{\infty}_{13}}+s^{\infty}_{12}\,\overline{s^{\infty}_{23}})}{e^{2ikL}-\overline{s^{\infty}_{33}}}\\[14pt]
=&s^{\infty}_{11}\,\overline{s^{\infty}_{12}}+s^{\infty}_{12}\,\overline{s^{\infty}_{22}}+s^{\infty}_{13}\,\overline{s^{\infty}_{23}}\,\cfrac{1-|s^{\infty}_{33}|^2}{|e^{-2ikL}-s^{\infty}_{33}|^2} -\cfrac{s^{\infty}_{13}\,s^{\infty}_{33}\,\overline{s^{\infty}_{23}}}{e^{-2ikL}-s^{\infty}_{33}}-\cfrac{\overline{s^{\infty}_{23}}\,s^{\infty}_{13}\,\overline{s^{\infty}_{33}}}{e^{2ikL}-\overline{s^{\infty}_{33}}}\\[14pt]
=& s^{\infty}_{11}\,\overline{s^{\infty}_{12}}+s^{\infty}_{12}\,\overline{s^{\infty}_{22}}+s^{\infty}_{13}\,\overline{s^{\infty}_{23}}\ =\ 0.
\end{array}
\]
Since $\mathbb{S}^{\mrm{asy}}(L)$ is symmetric, this is sufficient to conclude that $\mathbb{S}^{\mrm{asy}}(L)$ is unitary.
\end{proof}

\subsection{Particular cases in the asymptotic analysis of the scattering matrices}\label{paragraphDegenerated}

When $|s^{\infty}_{33}|=1$, the asymptotic (\ref{defCoeffLim}) of the scattering matrix $\mathbb{S}$ with respect to $L$ is different from what has been obtained above. More precisely, if $|s^{\infty}_{33}|=1$, since $\mathbb{S}^{\infty}$ is unitary and symmetric, then $s^{\infty}_{31}=s^{\infty}_{32}=0$, $s^{\infty}_{13}=s^{\infty}_{23}=0$. In such a situation, we can show that as $L\to+\infty$, the matrix $\mathbb{S}=\mathbb{S}(L)$ defined in (\ref{DefMatrixScattering}) tends to 
\[
\left(\begin{array}{cc}
s^{\infty}_{11} & s^{\infty}_{12}\\
s^{\infty}_{21} & s^{\infty}_{22}
\end{array}\right)\in\mathbb{C}^{2\times 2}.
\]
A similar phenomenon appears in the asymptotic (\ref{AsymptoticAugm}) of $\mathcal{S}$ when $|S^{\infty}_{33}|=1$. In this situation, we have $S^{\infty}_{31}=S^{\infty}_{32}=0$, $S^{\infty}_{13}=S^{\infty}_{23}=0$ and one can prove that $\mathcal{S}$ tends to 
\[
\left(\begin{array}{cc}
S^{\infty}_{11} & S^{\infty}_{12}\\
S^{\infty}_{21} & S^{\infty}_{22}
\end{array}\right)\in\mathbb{C}^{2\times 2}.
\]
In this article, we do not consider these exceptional cases which are not interesting for our analysis.

\bibliography{Bibli}
\bibliographystyle{plain}

\end{document}